\documentclass{elsarticle}
\usepackage{amsmath,amsfonts, amssymb, epsfig, amscd, graphicx, latexsym, latexsym, amsthm, etex}
\usepackage[latin5]{inputenc}
\usepackage{young}

\newcommand{\beeq}{\begin{equation}}
\newcommand{\eneq}{\end{equation}}
\newcommand{\bece}{\begin{center}}
\newcommand{\ence}{\end{center}}

\newtheorem{thm}{Theorem}[section]
\newtheorem{pro}[thm]{Proposition}
\newtheorem{cor}[thm]{Corollary}
\newtheorem{lem}[thm]{Lemma}

\newtheorem{defn}[thm]{Definition}

\newenvironment{ex}[1][Example.]{\begin{trivlist}
\item[\hskip \labelsep {\bfseries #1}]}{\end{trivlist}}
\newenvironment{rem}[1][Remark.]{\begin{trivlist}
\item[\hskip \labelsep {\bfseries #1}]}{\end{trivlist}}

\newcommand{\cellsize}{10}
\newlength{\cellsz} 
\newsavebox{\cell}
\sbox{\cell}{\begin{picture}(\cellsize,\cellsize)
\put(0,0){\line(1,0){\cellsize}}
\put(0,0){\line(0,1){\cellsize}}
\put(\cellsize,0){\line(0,1){\cellsize}}
\put(0,\cellsize){\line(1,0){\cellsize}}
\end{picture}}
\newcommand\cellify[1]{\def\thearg{#1}\def\nothing{}%
\ifx\thearg\nothing
\vrule width0pt height\cellsz depth0pt\else
\hbox to 0pt{\usebox{\cell} \hss}\fi%
\vbox to \cellsz{
\vss
\hbox to \cellsz{\hss$#1$\hss}
\vss}}
\newcommand\tableau[1]{\vtop{\let\\\cr
\baselineskip -16000pt \lineskiplimit 16000pt \lineskip 0pt
\ialign{&\cellify{##}\cr#1\crcr}}}

\begin{document}
\begin{frontmatter}

\title{Tower Tableaux}
\author{Olcay Co\c{s}kun}
\ead{olcay.coskun@boun.edu.tr}

\author{Müge Ta\c{s}k\i n}
\ead{muge.taskin@boun.edu.tr}
\address{Boğaziçi Üniversitesi, Matematik Bölümü 34342 Bebek, İstanbul, Turkey.}

\fntext[fn1]{Both of the authors are supported by Tübitak through Kariyer Program\i, (Tübitak/3501/109T661).}

\begin{abstract}
We introduce a new combinatorial object called tower diagrams and prove
fundamental properties of these objects. We also introduce an algorithm that allows us
to slide words to tower diagrams. We show that the algorithm is well-defined only for
reduced words which makes the algorithm a test for reducibility. Using the algorithm,
a bijection between tower diagrams and finite permutations is obtained and it is
shown that this bijection specializes to a bijection between certain labellings of
a given tower diagram and reduced expressions of the corresponding permutation.
\end{abstract}
\end{frontmatter}

\section{Introduction}
It is well-known that the symmetric group $S_n$ on $n$ objects is
generated by the set $\{ s_1,s_2, \ldots, s_{n-1} \}$ of all
adjacent transpositions, with the usual notation. A product of
these generators is called a word. A basic problem raised via this
observation is to determine the set of words which, when
multiplied, give the same permutation in $S_n$. A simple reduction
is obtained by considering the set of \textit{reduced} words,
words containing the minimal number of transpositions giving the
permutation.

The study of reduced words is initiated by a work of Stanley \cite{S} where he
proved a formula for the number of reduced words corresponding to the longest
permutation. After Stanley's algebraic proof,
there appear several other combinatorial proofs by Edelman-Greene \cite{EG} and
Lascoux-Schützenberger \cite{LS}.

There are also works that generalize this result of Stanley to an arbitrary
permutation. Some of which use balanced tableaux \cite{FGRS} of Edelman-Greene, RC-
graphs \cite{BB1} and plactification map \cite{RS}, see also \cite{BB}.

On the other hand, our approach is based on \textit{tower
diagrams}. By a tower diagram, we mean a diagram in the first
quadrant of the plane that consists of finitely many vertical strips with
bottoms on the $x$-axis, see Section \ref{Section:towerdiagrams} for a precise definition.
An example of a tower diagram is shown below.

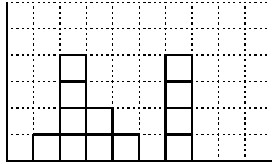
\begin{center}
\begin{picture}(100,60)
\multiput(0,0)(0,0){1}{\line(1,0){100}}
\multiput(0,0)(0,0){1}{\line(0,1){60}}
\multiput(0,10)(2,0){50}{\line(0,1){.1}}
\multiput(0,20)(2,0){50}{\line(0,1){.1}}
\multiput(0,30)(2,0){50}{\line(0,1){.1}}
\multiput(0,40)(2,0){50}{\line(0,1){.1}}
\multiput(0,50)(2,0){50}{\line(0,1){.1}}
\multiput(0,60)(2,0){50}{\line(0,1){.1}}
\multiput(10,0)(0,2){30}{\line(1,0){.1}}
\multiput(20,0)(0,2){30}{\line(1,0){.1}}
\multiput(30,0)(0,2){30}{\line(1,0){.1}}
\multiput(40,0)(0,2){30}{\line(1,0){.1}}
\multiput(50,0)(0,2){30}{\line(1,0){.1}}
\multiput(60,0)(0,2){30}{\line(1,0){.1}}
\multiput(70,0)(0,2){30}{\line(1,0){.1}}
\multiput(80,0)(0,2){30}{\line(1,0){.1}}
\multiput(90,0)(0,2){30}{\line(1,0){.1}}
\multiput(80,0)(0,2){30}{\line(1,0){.1}}
\multiput(90,0)(0,2){30}{\line(1,0){.1}} \put(10,0){\tableau{{}}}
\put(20,30){\tableau{{}\\{}\\{}\\{}}} \put(30,10){\tableau{{}\\{}}}
\put(40,0){\tableau{{}}} \put(60,30){\tableau{{}\\{}\\{}\\{}}}
\end{picture}
\end{center}
The results of the paper can be summarized via the following diagram.

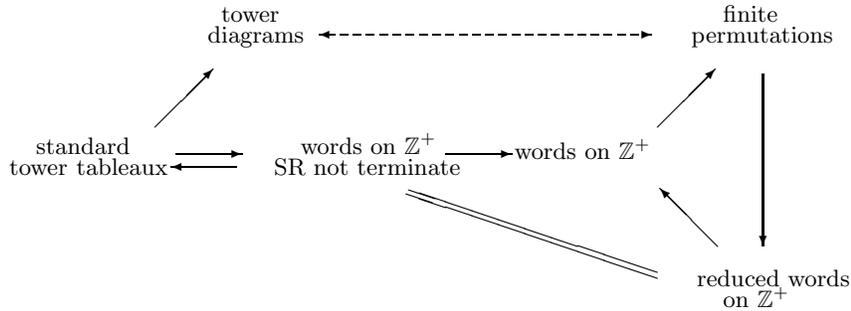
\begin{center}
\begin{picture}(160,120)

\put(-10,90){\small{tower}}
\put(-15,82){\small{diagrams}}

\put(180,90){\small{finite}}
\put(168,82){\small{permutations}}

\put(-80,40){\small{standard}}
\put(-90,32){\small{tower tableaux}}

\put(20,40){\small{words on $\mathbb Z^+$}}
\put(10,32){\small{SR not terminate}}

\put(101,38){\small{words on $\mathbb Z^+$}}

\put(170,-10){\small{reduced words}}
\put(180,-18){\small{on $\mathbb Z^+$}}

\multiput(30,85)(5,0){25}{\line(1,0){3}}
\multiput(30,85)(5,0){1}{\vector(-1,0){3}}
\multiput(150,85)(5,0){1}{\vector(1,0){3}}

\multiput(195,70)(0,0){1}{\vector(0,-1){65}}

\multiput(155,50)(0,0){1}{\vector(1,1){22}}

\multiput(178,5)(0,0){1}{\vector(-1,1){22}}

\multiput(155,-5)(0,0){1}{\line(-3,1){94}}
\multiput(155,-7)(0,0){1}{\line(-3,1){95}}


\multiput(-35,50)(0,0){1}{\vector(1,1){22}}

\multiput(-27,40)(0,0){1}{\vector(1,0){25}}
\multiput(-4,35)(0,0){1}{\vector(-1,0){25}}

\multiput(75,40)(0,0){1}{\vector(1,0){25}}
\end{picture}
\end{center}
\vspace{20pt}

Next, we provide an explanation of the above diagram. In Section
\ref{Section:towerdiagrams}, we first define a special labelling of
tower diagrams, called standard tower tableau. Then we define the
sliding and recording algorithm (SR algorithm, for short) on all
finite words over $\mathbb Z^+$, not necessarily reduced. This
algorithm lets us slide words to the plane with the $x$-axis being
the border, on reverse diagonal lines, subject to certain
conditions. As a result, when the algorithm terminates with a
result, we obtain a standard tower tableaux corresponding to the
given word. Conversely, we introduce a reading function which reads
a word on $\mathbb Z^+$ from each standard tower tableau.

In Section \ref{Section:SlidingIntoTower}, we prove that the SR algorithm does not
terminate if and only if the word is reduced, which gives us the equality seen above.
Therefore we obtain our first main result that there is a bijection from the set of
all reduced words to the set of all standard tower tableaux given by the reading
function and the SR algorithm. Moreover with this result, the sliding
algorithm becomes an algorithm which also tests if a given word is reduced or not.
Another algorithm that can be used to check reducibility is introduced by Edelmann-Greene
in \cite{EG}. They use a generalized RSK-algorithm to associate a pair of tableaux to any word. Then the
reducibility is detected by certain conditions on the tableaux
and, in some cases,
one needs to use braid relations to check reducibility of a certain tableau word.

The second main result of the paper establishes a connection between tower diagrams and permutations, shown by dashed arrows in the diagram. This
is done by showing that the sliding algorithm
associates the same tower diagram to two different reduced words
if and only if the words are braid related, that is, they
correspond to the same permutation, see Theorem \ref{thm:braid2shape}.
In particular,
any given tower diagram $\mathcal T$ determines a unique
permutation $\omega_{\mathcal T}$, and vice versa, any permutation
$\omega$ determines a unique tower diagram $\mathcal T_\omega$. Moreover
we establish an explicit bijective correspondence between
\begin{enumerate}
\item the set Red$(\omega)$ of reduced expressions of a given permutation $\omega$ and
\item the set STT$(\mathcal T_\omega)$ of standard tower tableaux of shape
$\mathcal T_\omega$.
\end{enumerate}
The next question is to determine the set Red$(\omega)$ using the
above bijection. It is possible to describe all standard tower
tableaux of a given shape by a recursive algorithm. We describe
this in Section \ref{Section:recordAndReduced}. Then the reduced words corresponding
to a given set of standard tower tableaux is given by the reading function. However,
the algorithm, being recursive, is slow. A faster and systematic algorithm that uses
tower tableaux will be introduced in a sequel to this paper.

About the determination of the cardinality of Red$(\omega)$, we
prove, in Sections \ref{section:rothe2tower} and
\ref{section:tower2rothe}, that our construction can be used
to determine the Rothe diagram of the permutation $\omega$ and
vice versa the Rothe diagram determines the tower diagram of
$\omega$. Therefore the above cardinality can be evaluated by
using the techniques in \cite{RS}. A remark about this
construction is that although it is straightforward to determine
the tower diagram from the given Rothe diagram, the converse is
tricky. In order to determine the Rothe diagram, we associate a
\textit{virtual} tower diagram to the permutation and show that
together with the tower diagram, the virtual tower diagram
recovers the Rothe diagram.

However this observation does not mean that the two constructions,
tower diagrams and Rothe diagrams, are equivalent. A trivial
observation is that any tower diagram corresponds
to a permutation. On the other hand, Rothe diagrams can not be
chosen arbitrarily. Another important feature of tower diagrams is
that they unearth certain information regarding the reduced words
that cannot be read from the Rothe diagram.

An example of such an
information is the existence of the \textit{natural word} of a
permutation, introduced in Section \ref{Section:NaturalWord}. The natural word for a permutation is the reduced
word which consists of (strictly) increasing subsequences of
consecutive integers in which the subsequence of the initial terms
of these sequences is strictly decreasing. This word comes
canonically with the the associated tower diagram. As far as we know, no special attention was paid to the natural word previously.
This canonical word is already used crucially in Section
\ref{section:tower2rothe} in relation with
the problem of determination of the Rothe diagram from the tower diagram.  More importantly, it
will be the main tool in the construction of the above mentioned algorithm in a
sequel where we also introduce a variation of the selection sort algorithm on reduced words. See \cite{K} for the
selection sort algorithm.
 
\textbf{Acknowledgement.} We thank the referees for their helpful remarks.The second author thanks Alex Young, 
Hugh Thomas and Nanthel Bergeron for their helpful conversations and remarks on the subject during her visit at 
Field Institiute in 2008. 

\section{Tower Diagrams and Tower Tableaux}
\label{Section:towerdiagrams}
In this section, we introduce tower tableaux together with their
basic properties. First, recall that
a sequence of non-negative integers $\tau=(\tau_1,\tau_2,\ldots,\tau_k)$ is
called a \textbf{\textit{weak composition}} of $n$ if each term $\tau_k$ of the sequence is non-negative and they sum up to $n$.

By a \textbf{\textit {tower}} $\mathcal{T}$ of size $k\ge 0$ we mean a vertical strip
of $k$ cell cells of side length $1$. On the other hand, a sequence $\mathcal T
 =(\mathcal{T}_1,\mathcal{T}_2,\ldots)$ of towers in which only finitely many towers
has positive size is called a \textbf{\textit{tower diagram}}. We always consider
the tower diagram $\mathcal T$ as located on the first quadrant of the plane so that
for each $i$, the tower $\mathcal{T}_i$ is located on the interval $[i-1,i]$ of the
horizontal axis and has size equal to the size of $\mathcal T_i$.

Let $\mathcal T =(\mathcal T_1,\mathcal T_2,\ldots)$ be a tower diagram and let
$\mathcal T_i$ (resp. $\mathcal T_j$) be the first (resp. the last) tower of $\mathcal T$
with non-zero size. Then we abbreviate $\mathcal T$ as $\mathcal T=
(\mathcal T_i,\ldots, \mathcal T_j)$.

Now let $\tau_i$ denote the size of the tower $\mathcal T_i$. It is clear that the
sequence $\tau=(\tau_i, \ldots,\tau_j)$ is a weak composition of the size of $\mathcal T$. Here the size of a tower diagram is defined by the sum of the sizes of its towers.
Conversely, it is natural to represent a
weak composition $\tau=(\tau_1,\tau_2,\ldots,\tau_k)$ by a tower diagram
which consist of a sequence of towers
$\mathcal{T}=(\mathcal{T}_1,\mathcal{T}_2,\ldots, \mathcal{T}_k)$ with the size of
$\mathcal T_i$ equal to $\tau_i$.

To any tower diagram $\mathcal{T}$, one can associate a set, still denoted by
$\mathcal{T}$, consisting of the pairs of non-negative integers with
the rule  that each pair $(i,j)$ corresponds to the cell in
$\mathcal{T}$ whose south-east corner is located at the point $(i,j)$ of the
first quadrant. Such a set can also be characterized by the  rule that if
$(i,j)\in \mathcal{T}$ then $\{(i,0), (i,1),\ldots, (i,j)
\}\subset \mathcal{T}$. For the rest, we identify any cell with its south-east
corner.

\begin{ex} For the weak composition $\tau=(0,1,4,2,1,0,4)$ the
corresponding tower diagram  $\mathcal{T}$ and the corresponding set
$\mathcal{T}$ are given as follows.

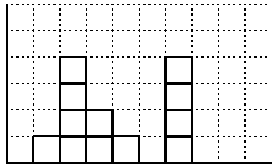
\begin{center}
\begin{picture}(100,60)
\multiput(0,0)(0,0){1}{\line(1,0){100}}
\multiput(0,0)(0,0){1}{\line(0,1){60}}
\multiput(0,10)(2,0){50}{\line(0,1){.1}}
\multiput(0,20)(2,0){50}{\line(0,1){.1}}
\multiput(0,30)(2,0){50}{\line(0,1){.1}}
\multiput(0,40)(2,0){50}{\line(0,1){.1}}
\multiput(0,50)(2,0){50}{\line(0,1){.1}}
\multiput(0,60)(2,0){50}{\line(0,1){.1}}
\multiput(10,0)(0,2){30}{\line(1,0){.1}}
\multiput(20,0)(0,2){30}{\line(1,0){.1}}
\multiput(30,0)(0,2){30}{\line(1,0){.1}}
\multiput(40,0)(0,2){30}{\line(1,0){.1}}
\multiput(50,0)(0,2){30}{\line(1,0){.1}}
\multiput(60,0)(0,2){30}{\line(1,0){.1}}
\multiput(70,0)(0,2){30}{\line(1,0){.1}}
\multiput(80,0)(0,2){30}{\line(1,0){.1}}
\multiput(90,0)(0,2){30}{\line(1,0){.1}}
\multiput(80,0)(0,2){30}{\line(1,0){.1}}
\multiput(90,0)(0,2){30}{\line(1,0){.1}} \put(10,0){\tableau{{}}}
\put(20,30){\tableau{{}\\{}\\{}\\{}}} \put(30,10){\tableau{{}\\{}}}
\put(40,0){\tableau{{}}} \put(60,30){\tableau{{}\\{}\\{}\\{}}}
\end{picture}
\end{center}
$$\mathcal{T}=\{(2,0), (3,3), (3,2), (3,1), (3,0), (4,1), (4,0), (5,0), (7,3), (7,2),$$ $$(7,1), (7,0)\}
$$
\end{ex}

Writing $(n)$ for the trivial weak composition, we can think of any weak composition
$\tau=(\tau_1,\tau_2,\ldots,\tau_k)$ as a concatenation of trivial weak compositions
$$
\tau = (\tau_1)\sqcup (\tau_2)\sqcup \ldots \sqcup (\tau_k).
$$
In a similar way, we can regard any tower diagram $\mathcal T = (\mathcal T_i,
\ldots,\mathcal T_j)$ as a concatenation of its towers and write
$$
\mathcal T = (\mathcal T_i)\sqcup (\mathcal T_{i+1})\sqcup \ldots \sqcup (\mathcal T_j).
$$
It is straightforward that the concatenation of towers can be generalized to the
concatenation of two tower diagrams provided that one tower lies completely on the
right of the other one.

A basic operation to obtain new tower diagrams from old is to let some cells fly
from the diagram. The reason for this operation will become clear later when
we introduce the reading word of a labelled diagram. We define the flight as follows.

\begin{defn} \label{def:flight} Let $(i,j)$ be a cell  in $\mathcal{T}$.
\begin{enumerate}
\item[i)] The cell $(i,j)$ \textbf{\textit{lies on the diagonal line}} $x+y=d$ if
its main diagonal is a part of $x+y=d$, that is, if $d=i+j$.

\item[ii)] The cell $(i,j)$ is said to have a \textbf{flight
path} in $\mathcal{T}$ if one of the following conditions is satisfied:

\begin{enumerate}
    \item[\textbf{(F1)}] (Direct flight) The diagram $\mathcal{T}$
has no cell to the left of $(i,j)$ lying on (and therefore above)
the diagonal $x+y=i+j-1$. In this case the flight path of $(i,j)$ in
$\mathcal T$ is defined by
$$\mathrm{flightpath}((i,j),\mathcal T):=\{(i,j)\}.
$$
\item[\textbf{(F2)}] (Zigzag flight) Among all cells of the diagram $\mathcal{T}$ lying on the diagonal
    $x+y=i+j-1$ and to the left of $(i,j)$,  the one closest to $(i,j)$, say
    $(i',j')$, has a flight path and  $(i',j'+1) \in \mathcal{T} $.  In this case the flight path of $(i,j)$ in $\mathcal T$ is defined by
$$\mathrm{flightpath}((i,j),\mathcal T):=\{(i,j),
(i',j'+1)\}\cup\mathrm{flightpath}((i',j'),\mathcal T).$$
\end{enumerate}

\item[iii)] If $(i,j)$ has  a flight path in $\mathcal T$, let $(i^\prime,j^\prime)$ be
the minimum element in the flight path of $(i,j)$ with respect to
the lexicographic order.
 Then the number  $i^\prime +j^\prime$ is called the \textbf{\textit{flight number}} of the cell $(i,j)$, denoted by
$$\mathrm{flight}\#((i,j),\mathcal T).$$
\item[iv)] The cell $(i,j)$ is called a \textbf{corner cell} of $\mathcal{T}$, if
$(i,j+1) \not\in \mathcal{T}$ and $(i,j)$ has a flight path.
\item[v)] Let $c = (i,j)$ be a corner cell in $\mathcal T$.  The tower diagram obtained from $\mathcal T$ by removing
the corner cell $c$ is denoted by
$$c^{\nwarrow}\mathcal{T}.$$
\end{enumerate}
\end{defn}
\begin{rem}
\begin{enumerate}
\item One can easily observe that if two cells have the same flight number in $\mathcal T=(\mathcal T_1,
\mathcal T_2,\mathcal T_3,\ldots)$ then they have the same flight number in  $(\mathcal T_2,\mathcal T_3,\ldots)$.
Thus if $(i,j)$ and $(i',j')$ are two cells in $\mathcal T$ with $(i',j')$ is
lexicographically smaller, then  both cells have the same flight number if and only if  both $(i',j')$ and  $(i',j'+1)$ lies
in $\mathrm{flightpath} ((i,j),\mathcal T)$.
\item We sometimes consider the flight path of a cell $(i,j)$ as the trace of the south-
east corner of $(i,j)$ on the plane. Hence by a flight path, we mean a zigzag line as
seen in the examples below.
\end{enumerate}
\end{rem}

\begin{ex} We will consider the tower diagram $\mathcal T$ which corresponds to  the weak
composition $\tau=(0,1,4,2,1,0,4)$ of the previous example. We
first show that the only cells without a flight path are $(3,0)$,
$(4,0)$, $(5,0)$, and $(7,0)$, as the following diagrams illustrate
respectively.

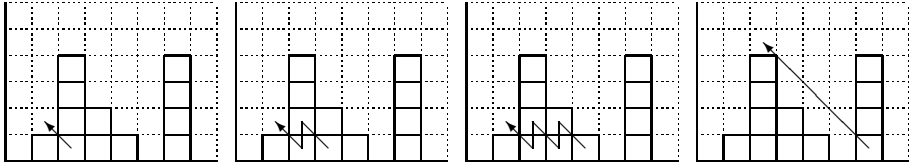
\begin{center}
\begin{picture}(80,60)
\multiput(0,0)(0,0){1}{\line(1,0){80}}
\multiput(0,0)(0,0){1}{\line(0,1){60}}
\multiput(0,10)(2,0){40}{\line(0,1){.1}}
\multiput(0,20)(2,0){40}{\line(0,1){.1}}
\multiput(0,30)(2,0){40}{\line(0,1){.1}}
\multiput(0,40)(2,0){40}{\line(0,1){.1}}
\multiput(0,50)(2,0){40}{\line(0,1){.1}}
\multiput(0,60)(2,0){40}{\line(0,1){.1}}
\multiput(10,0)(0,2){30}{\line(1,0){.1}}
\multiput(20,0)(0,2){30}{\line(1,0){.1}}
\multiput(30,0)(0,2){30}{\line(1,0){.1}}
\multiput(40,0)(0,2){30}{\line(1,0){.1}}
\multiput(50,0)(0,2){30}{\line(1,0){.1}}
\multiput(60,0)(0,2){30}{\line(1,0){.1}}
\multiput(70,0)(0,2){30}{\line(1,0){.1}}
\multiput(80,0)(0,2){30}{\line(1,0){.1}}
 \put(10,0){\tableau{{}}}
\put(20,30){\tableau{{}\\{}\\{}\\{}}} \put(30,10){\tableau{{}\\{}}}
\put(40,0){\tableau{{}}} \put(60,30){\tableau{{}\\{}\\{}\\{}}}
\multiput(25,5)(0,0){1}{\vector(-1,1){10}}
\end{picture}\hskip.1in
\begin{picture}(80,60)
\multiput(0,0)(0,0){1}{\line(1,0){80}}
\multiput(0,0)(0,0){1}{\line(0,1){60}}
\multiput(0,10)(2,0){40}{\line(0,1){.1}}
\multiput(0,20)(2,0){40}{\line(0,1){.1}}
\multiput(0,30)(2,0){40}{\line(0,1){.1}}
\multiput(0,40)(2,0){40}{\line(0,1){.1}}
\multiput(0,50)(2,0){40}{\line(0,1){.1}}
\multiput(0,60)(2,0){40}{\line(0,1){.1}}
\multiput(10,0)(0,2){30}{\line(1,0){.1}}
\multiput(20,0)(0,2){30}{\line(1,0){.1}}
\multiput(30,0)(0,2){30}{\line(1,0){.1}}
\multiput(40,0)(0,2){30}{\line(1,0){.1}}
\multiput(50,0)(0,2){30}{\line(1,0){.1}}
\multiput(60,0)(0,2){30}{\line(1,0){.1}}
\multiput(70,0)(0,2){30}{\line(1,0){.1}}
\multiput(80,0)(0,2){30}{\line(1,0){.1}}
 \put(10,0){\tableau{{}}}
\put(20,30){\tableau{{}\\{}\\{}\\{}}} \put(30,10){\tableau{{}\\{}}}
\put(40,0){\tableau{{}}} \put(60,30){\tableau{{}\\{}\\{}\\{}}}
\multiput(35,5)(0,0){1}{\line(-1,1){10}}
\multiput(25,15)(0,0){1}{\line(0,-1){10}}
\multiput(25,5)(0,0){1}{\vector(-1,1){10}}
\end{picture}\hskip.1in
\begin{picture}(80,60)
\multiput(0,0)(0,0){1}{\line(1,0){80}}
\multiput(0,0)(0,0){1}{\line(0,1){60}}
\multiput(0,10)(2,0){40}{\line(0,1){.1}}
\multiput(0,20)(2,0){40}{\line(0,1){.1}}
\multiput(0,30)(2,0){40}{\line(0,1){.1}}
\multiput(0,40)(2,0){40}{\line(0,1){.1}}
\multiput(0,50)(2,0){40}{\line(0,1){.1}}
\multiput(0,60)(2,0){40}{\line(0,1){.1}}
\multiput(10,0)(0,2){30}{\line(1,0){.1}}
\multiput(20,0)(0,2){30}{\line(1,0){.1}}
\multiput(30,0)(0,2){30}{\line(1,0){.1}}
\multiput(40,0)(0,2){30}{\line(1,0){.1}}
\multiput(50,0)(0,2){30}{\line(1,0){.1}}
\multiput(60,0)(0,2){30}{\line(1,0){.1}}
\multiput(70,0)(0,2){30}{\line(1,0){.1}}
\multiput(80,0)(0,2){30}{\line(1,0){.1}}
 \put(10,0){\tableau{{}}}
\put(20,30){\tableau{{}\\{}\\{}\\{}}} \put(30,10){\tableau{{}\\{}}}
\put(40,0){\tableau{{}}} \put(60,30){\tableau{{}\\{}\\{}\\{}}}
\multiput(45,5)(0,0){1}{\line(-1,1){10}}
\multiput(35,15)(0,0){1}{\line(0,-1){10}}
\multiput(35,5)(0,0){1}{\line(-1,1){10}}
\multiput(25,15)(0,0){1}{\line(0,-1){10}}
\multiput(25,5)(0,0){1}{\vector(-1,1){10}}
\end{picture}\hskip.1in
\begin{picture}(80,60)
\multiput(0,0)(0,0){1}{\line(1,0){80}}
\multiput(0,0)(0,0){1}{\line(0,1){60}}
\multiput(0,10)(2,0){40}{\line(0,1){.1}}
\multiput(0,20)(2,0){40}{\line(0,1){.1}}
\multiput(0,30)(2,0){40}{\line(0,1){.1}}
\multiput(0,40)(2,0){40}{\line(0,1){.1}}
\multiput(0,50)(2,0){40}{\line(0,1){.1}}
\multiput(0,60)(2,0){40}{\line(0,1){.1}}
\multiput(10,0)(0,2){30}{\line(1,0){.1}}
\multiput(20,0)(0,2){30}{\line(1,0){.1}}
\multiput(30,0)(0,2){30}{\line(1,0){.1}}
\multiput(40,0)(0,2){30}{\line(1,0){.1}}
\multiput(50,0)(0,2){30}{\line(1,0){.1}}
\multiput(60,0)(0,2){30}{\line(1,0){.1}}
\multiput(70,0)(0,2){30}{\line(1,0){.1}}
\multiput(80,0)(0,2){30}{\line(1,0){.1}}
 \put(10,0){\tableau{{}}}
\put(20,30){\tableau{{}\\{}\\{}\\{}}} \put(30,10){\tableau{{}\\{}}}
\put(40,0){\tableau{{}}} \put(60,30){\tableau{{}\\{}\\{}\\{}}}
\multiput(65,5)(0,0){1}{\vector(-1,1){40}}
\end{picture}
\end{center}

First observe that the cell $(2,0)$ is the closest cell to $(3,0)$ lying the diagonal
$x+y=2$, to the left of $(3,0)$. One can easily see that the cell $(2,0)$ has a
flight path which consists only of the cell $(2,0)$ itself. On the
other hand, $(2,1)$ does not belong to $\mathcal T$ and therefore
$(3,0)$ has no flight path in  $\mathcal T$. Similar reasoning also
applies to the cell $(7,0)$.

The closest cell lying to the left of $(4,0)$ on the diagonal
$x+y=3$ is $(3,0)$. Now if  $(4,0)$ has a flight path then it must
contain the cell $(3,1)$ and the flight path of $(3,0)$. On the
other hand although $(3,1)$ belongs to $\mathcal T$, the cell
$(3,0)$ has no flight path, so  $(4,0)$  has  no flight path in
$\mathcal T$. A similar argument applied on the cell $(5,0)$
shows that it has no flight path in $\mathcal T$.

We now illustrate the flight paths  of $(3,1)$ and
$(4,1)$ by the following diagrams.

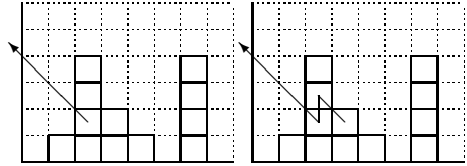
\begin{center}
\begin{picture}(80,60)
\multiput(0,0)(0,0){1}{\line(1,0){80}}
\multiput(0,0)(0,0){1}{\line(0,1){60}}
\multiput(0,10)(2,0){40}{\line(0,1){.1}}
\multiput(0,20)(2,0){40}{\line(0,1){.1}}
\multiput(0,30)(2,0){40}{\line(0,1){.1}}
\multiput(0,40)(2,0){40}{\line(0,1){.1}}
\multiput(0,50)(2,0){40}{\line(0,1){.1}}
\multiput(0,60)(2,0){40}{\line(0,1){.1}}
\multiput(10,0)(0,2){30}{\line(1,0){.1}}
\multiput(20,0)(0,2){30}{\line(1,0){.1}}
\multiput(30,0)(0,2){30}{\line(1,0){.1}}
\multiput(40,0)(0,2){30}{\line(1,0){.1}}
\multiput(50,0)(0,2){30}{\line(1,0){.1}}
\multiput(60,0)(0,2){30}{\line(1,0){.1}}
\multiput(70,0)(0,2){30}{\line(1,0){.1}}
\multiput(80,0)(0,2){30}{\line(1,0){.1}}
 \put(10,0){\tableau{{}}}
\put(20,30){\tableau{{}\\{}\\{}\\{}}} \put(30,10){\tableau{{}\\{}}}
\put(40,0){\tableau{{}}} \put(60,30){\tableau{{}\\{}\\{}\\{}}}
\multiput(25,15)(0,0){1}{\vector(-1,1){30}}
\end{picture}\hskip.1in
\begin{picture}(80,60)
\multiput(0,0)(0,0){1}{\line(1,0){80}}
\multiput(0,0)(0,0){1}{\line(0,1){60}}
\multiput(0,10)(2,0){40}{\line(0,1){.1}}
\multiput(0,20)(2,0){40}{\line(0,1){.1}}
\multiput(0,30)(2,0){40}{\line(0,1){.1}}
\multiput(0,40)(2,0){40}{\line(0,1){.1}}
\multiput(0,50)(2,0){40}{\line(0,1){.1}}
\multiput(0,60)(2,0){40}{\line(0,1){.1}}
\multiput(10,0)(0,2){30}{\line(1,0){.1}}
\multiput(20,0)(0,2){30}{\line(1,0){.1}}
\multiput(30,0)(0,2){30}{\line(1,0){.1}}
\multiput(40,0)(0,2){30}{\line(1,0){.1}}
\multiput(50,0)(0,2){30}{\line(1,0){.1}}
\multiput(60,0)(0,2){30}{\line(1,0){.1}}
\multiput(70,0)(0,2){30}{\line(1,0){.1}}
\multiput(80,0)(0,2){30}{\line(1,0){.1}}
 \put(10,0){\tableau{{}}}
\put(20,30){\tableau{{}\\{}\\{}\\{}}} \put(30,10){\tableau{{}\\{}}}
\put(40,0){\tableau{{}}} \put(60,30){\tableau{{}\\{}\\{}\\{}}}
\multiput(35,15)(0,0){1}{\line(-1,1){10}}
\multiput(25,25)(0,0){1}{\line(0,-1){10}}
\multiput(25,15)(0,0){1}{\vector(-1,1){30}}
\end{picture}
\end{center}

Here the flight paths are given by
$$
\begin{aligned}&\mathrm{flightpath}((3,1),\mathcal T)=\{(3,1)\}\\
&\mathrm{flightpath}((4,1),\mathcal T)=\{(4,1),(3,2),(3,1)\}
\end{aligned}$$
whereas corresponding flight numbers are both equal to $4$. Moreover
$(4,1)$ is a corner cell in $\mathcal T$ since it is also a top
cell.

One can easily check that the remaining cells in $\mathcal T$
also have flight paths. On the other hand, among all of them, the cells
$(2,0),(3,3), (4,1)$ and $(7,3)$ are the corner cells in $\mathcal T$.
\end{ex}

In the next lemma, we examine the relation between consecutive flights of cells.
This technical lemma will be used in Section \ref{Section:recordAndReduced}. We
postpone the proof of this lemma to \ref{Appendix}.

\begin{lem}\label{lem:ax-zigzag}
Let $\mathcal T$ be a tower diagram and $c_1 = (i,j_1)$ and $c_2 = (i,j_2)$ be two
cells in the tower $\mathcal T_i$ of $\mathcal T$. Assume that $c_1$ and $c_2$ have
flight paths with respective flight numbers $f_1$ and $f_2$.
Then $|f_1-f_2|\ge |j_1-j_2|$. Moreover if $|j_1-j_2| = 1$ then $|f_1-f_2| = 1$.
\end{lem}

Since one of our aims is to relate labelled tower diagrams to
words, next we specify the kinds of labelling that will be used through
the rest of the paper.
\begin{defn} Let $\mathcal{T}$ be a tower diagram of size $n$ and
  $f:\mathcal{T}\mapsto [n]$ be a bijective map. Here we put  $[n] =
\{1,2,\ldots, n \}$.
Then
\begin{enumerate}
\item[i)] The set $$T=\{[(i,j),f(i,j)]\mid (i,j)\in \mathcal{T}\}$$ is called a
\textbf{tower tableau of shape $\mathcal{T}$}. In this case we write $\mathrm{shape}(T)=\mathcal{T}.$
\item[ii)] Given a tower tableau  $T$ of size $n$ and $a\in [n]$, the set
$$T_{\leq a}:=\{ [(i,j),b]\in T \mid b\leq a\}.
$$
is a (not necessarily tower) subtableau of cells in $T$ whose
labels are less then or equal to $a$.
\item[iii)] A tower tableau $T$ of shape $\mathcal{T}$ is called a
\textbf{standard tower tableau} if for each $[(i,j),a]$ in $T$,
the tableau $T_{\le a}$ is a tower tableau and moreover the cell $(i,j)$ is
a corner cell of the diagram $\mathrm{shape}(T_{\leq a})$.
\item[iv)] The set of all standard tower tableaux of all shapes is denoted by STT.
\end{enumerate}
\end{defn}
\begin{ex} Let $\tau = (2,1,0,1)$. Then the standard tower tableaux of this shape
are given as follows.
$$\line(0,1){30}\line(1,0){50}
\multiput(-50,0)(10,0){5}{\line(0,1){2}}
\multiput(-50,0)(0,10){3}{\line(1,0){2}}
\put(-50,20){\tableau{\\{4}\\{3}}} \put(-40,10){\tableau{\\{2}}}
\put(-20,10){\tableau{\\{1}}} \hskip.1in
\line(0,1){30}\line(1,0){50}
\multiput(-50,0)(10,0){5}{\line(0,1){2}}
\multiput(-50,0)(0,10){3}{\line(1,0){2}}
\put(-50,20){\tableau{\\{4}\\{2}}} \put(-40,10){\tableau{\\{1}}}
\put(-20,10){\tableau{\\{3}}} \hskip.1in
\line(0,1){30}\line(1,0){50}
\multiput(-50,0)(10,0){5}{\line(0,1){2}}
\multiput(-50,0)(0,10){3}{\line(1,0){2}}
\put(-50,20){\tableau{\\{4}\\{3}}} \put(-40,10){\tableau{\\{1}}}
\put(-20,10){\tableau{\\{2}}}
 \hskip.1in
\line(0,1){30}\line(1,0){50}
\multiput(-50,0)(10,0){5}{\line(0,1){2}}
\multiput(-50,0)(0,10){3}{\line(1,0){2}}
\put(-50,20){\tableau{\\{3}\\{2}}} \put(-40,10){\tableau{\\{4}}}
\put(-20,10){\tableau{\\{1}}} \hskip.1in
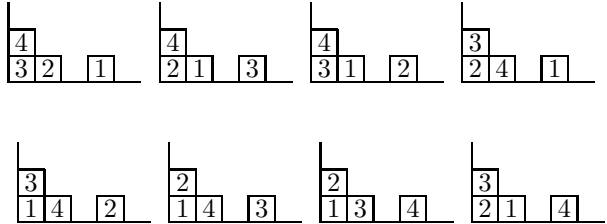
$$

$$
\line(0,1){30}\line(1,0){50}
\multiput(-50,0)(10,0){5}{\line(0,1){2}}
\multiput(-50,0)(0,10){3}{\line(1,0){2}}
\put(-50,20){\tableau{\\{3}\\{1}}} \put(-40,10){\tableau{\\{4}}}
\put(-20,10){\tableau{\\{2}}}
 \hskip.1in
 \line(0,1){30}\line(1,0){50}
\multiput(-50,0)(10,0){5}{\line(0,1){2}}
\multiput(-50,0)(0,10){3}{\line(1,0){2}}
\put(-50,20){\tableau{\\{2}\\{1}}} \put(-40,10){\tableau{\\{4}}}
\put(-20,10){\tableau{\\{3}}} \hskip.1in
\line(0,1){30}\line(1,0){50}
\multiput(-50,0)(10,0){5}{\line(0,1){2}}
\multiput(-50,0)(0,10){3}{\line(1,0){2}}
\put(-50,20){\tableau{\\{2}\\{1}}} \put(-40,10){\tableau{\\{3}}}
\put(-20,10){\tableau{\\{4}}} \hskip.1in
\line(0,1){30}\line(1,0){50}
\multiput(-50,0)(10,0){5}{\line(0,1){2}}
\multiput(-50,0)(0,10){3}{\line(1,0){2}}
\put(-50,20){\tableau{\\{3}\\{2}}} \put(-40,10){\tableau{\\{1}}}
\put(-20,10){\tableau{\\{4}}} $$ On the other hand, the labelling
$$\line(0,1){30}\line(1,0){50}
\multiput(-50,0)(10,0){5}{\line(0,1){2}}
\multiput(-50,0)(0,10){3}{\line(1,0){2}}
\put(-50,20){\tableau{\\{3}\\{1}}} \put(-40,10){\tableau{\\{2}}}
\put(-20,10){\tableau{\\{4}}} $$ of $(2,1,0,1)$ is not a standard
tower tableau since the cell labelled by $2$ is not a corner
cell of $T_{\leq 2}$.
\end{ex}
Next we introduce the \textbf{\textit{reading function}}
\[
\mathrm{Read}: \mathrm{STT}\rightarrow \mathrm{W}(\mathbb Z^+)
\]
from the set of all standard tower tableaux to the set $\mathrm{W}(\mathbb Z^+)$ of all finite words over $\mathbb Z^+$ as follows. This definition justifies the choice of the standard labellings defined above.

Let $R$ be a standard tower tableau of size $n$. Then for each
$k\in\{1,\ldots, n\}$ the cell labelled by $k$ in $R$, say
$(i_k,j_k)$, is a corner cell in $\mathrm{shape}(R_{\leq k})$
and therefore it has a flight path in $\mathrm{shape}(R_{\leq
k})$. We let
$$
\alpha_k=\mathrm{flight}\#((i_k,j_k),\mathrm{shape}(R_{\leq k})).
$$
One can easily see that if $(i_k,j_k)$ satisfies $(\textbf{F1})$
then $\alpha_k=i_k+j_k$ otherwise $\alpha_k=i_k+j_k-f_k$ where $f_k$
is the number of times that $(\textbf{F2})$ is used in the
construction of
$\mathrm{flightpath}((i_k,j_k),\mathrm{shape}(R_{\leq k}))$.
Finally let
$$ \mathrm{Read}(R):= \alpha_1\ldots \alpha_k \ldots \alpha_n.$$
We call the word $\mathrm{Read}(R)$ the \textbf{\textit{reading word}} of $R$.

\begin{ex}
The reading words of the standard tower tableaux given in the previous example are listed below.
\[
4212, 2142, 2412, 4121, 1421, 1241, 1214, 2124.
\]
\end{ex}

For the rest of the paper, basically, we analyse the reading function and
an inverse of it. We leave this function alone until the end of
Section \ref{Section:SlidingAlg}.

\section{Sliding and Recording Algorithm}
\label{Section:SlidingIntoTower}\label{Section:SlidingAlg}
The main tool in defining a function from the set $\mathrm{W}(\mathbb Z^+)$ of words
to the set $\mathrm{STT}$ of all standard tower tableaux is the
sliding and recording algorithm that we shall define in this section.

As a preparation to the definition, we first introduce the basic move for the
algorithm, called sliding into a tower diagram.
This is a way to enlarge a tower diagram by sliding a new cell into it. As one would expect, the new cell
will have a flight path which can be specified through sliding. We also prove a couple of lemmas to clarify
the relation between consecutive slides. In particular, we show that the slide operation satisfies braid
relations. We begin with the definition of the slide operation.

\begin{defn}\label{def:sliding} Let $\mathcal{T}=(\mathcal{T}_1,\mathcal{T}_2,\ldots)$
be a tower diagram and $\alpha$ be a positive integer. In the
following we denote   the \textbf{\textit{sliding of}} $\alpha$ into
$\mathcal{T}$ by
$$\alpha^{\searrow}
\mathcal{T}=\alpha^{\searrow}(\mathcal{T}_1,\mathcal{T}_2,\ldots).
$$
\begin{enumerate}
\item[\textbf{(S1)}] If  $\mathcal{T}$ has  no squares lying on the diagonal
$x+y=\alpha-1$ then we put
$$\alpha^{\searrow} \mathcal{T}:= (\mathcal{T}_1,\ldots,\mathcal{T}_{\alpha-1}) \sqcup \alpha^{\searrow}
(\mathcal{T}_{\alpha},\ldots)$$
\begin{enumerate}
\item If  $\mathcal{T}$ has  no squares lying on the diagonal
$x+y=\alpha$ then necessarily  $\mathcal{T}_\alpha =\varnothing$ and
for $\mathcal{T}_\alpha'= \{ (\alpha,0)\}$
$$\alpha^{\searrow}
(\mathcal{T}_{\alpha},\ldots)=(\mathcal{T}_{\alpha}',\ldots)~~\text{and}~~
\alpha^{\searrow} \mathcal{T}:=(\mathcal{T}_1\ldots
\mathcal{T}_{\alpha-1},\mathcal{T}'_\alpha,\mathcal{T}_{\alpha+1},\ldots
).
$$
\item If $(\alpha,0) \in \mathcal{T}_\alpha$ and $(\alpha,1)\not \in
\mathcal{T}_\alpha$ then the slide  $\alpha^{\searrow} \mathcal{T}$
terminates without a result.
\item If $(\alpha,0) \in \mathcal{T}_\alpha$ and $(\alpha,1) \in
\mathcal{T}_\alpha$  then
$$\alpha^{\searrow} \mathcal{T}:= (\mathcal{T}_1,\ldots,\mathcal{T}_\alpha) \sqcup
(\alpha+1)^{\searrow} (\mathcal{T}_{\alpha+1},\ldots).
$$
and $\alpha^{\searrow} \mathcal{T}$ terminates if and only if
$(\alpha+1)^{\searrow} (\mathcal{T}_{\alpha+1},\ldots)$ terminates.
\end{enumerate}
\item[\textbf{(S2)}]  Suppose now that  $\mathcal{T}$ has  some squares lying on the diagonal
$x+y=\alpha-1$ and  let $\mathcal{T}_i$ be the first tower  from the
left which contains such a  square, which is necessarily
$(i,\alpha-1-i)$ for some $1\leq i < \alpha$. Then we put
$$\alpha^{\searrow} \mathcal{T}:=
(\mathcal{T}_1,\ldots,\mathcal{T}_{i-1}) \sqcup \alpha^{\searrow}
(\mathcal{T}_{i},\ldots).$$
\begin{enumerate}
\item If $(i,\alpha-i)\not \in \mathcal{T}_i$ then for $\mathcal{T}_i'=\mathcal{T}_i \cup \{
(i,\alpha-i)\}$,
$$\alpha^{\searrow} (\mathcal{T}_{i},\ldots):=(\mathcal{T}_i',\ldots)~~\text{and}~~
\alpha^{\searrow} \mathcal{T}:=(\mathcal{T}_1\ldots
\mathcal{T}_{i-1},\mathcal{T}'_i,\mathcal{T}_{i+1},\ldots ).
$$
\item If $(i,\alpha-i) \in \mathcal{T}_i$ and $(i,\alpha-i+1)\not \in
\mathcal{T}_i$ then the slide $\alpha^{\searrow} \mathcal{T}$
terminates without a result.
\item If $(i,\alpha-i) \in \mathcal{T}_i$ and $(i,\alpha-i+1) \in
\mathcal{T}_i$  then
$$\alpha^{\searrow} \mathcal{T}:= (\mathcal{T}_1,\ldots,\mathcal{T}_i) \sqcup
(\alpha+1)^{\searrow} (\mathcal{T}_{i+1},\ldots)
$$
and $\alpha^{\searrow} \mathcal{T}$ terminates if and only if
$(\alpha+1)^{\searrow} (\mathcal{T}_{i+1},\ldots)$ terminates.
\end{enumerate}
\end{enumerate}

Therefore if the algorithm does not terminate then
$\alpha^{\searrow} \mathcal{T}:= \mathcal{T}\cup \{ (i,j)\}$ for
some square $(i,j)$.
\end{defn}

\begin{ex} Let
$\mathcal{T}=(\mathcal{T}_1,\mathcal{T}_2,\mathcal{T}_3,\mathcal{T}_4,\mathcal{T}_5,\mathcal{T}_6,\mathcal{T}_7)=(1,0,4,2,0,0,2)$ be the
tower diagram shown below.
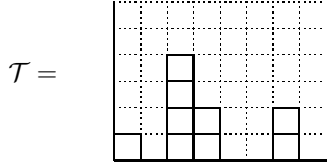
\begin{center}
\begin{picture}(80,60)
\put(-40,30){$\mathcal{T}= $}
\multiput(0,0)(0,0){1}{\line(1,0){80}}
\multiput(0,0)(0,0){1}{\line(0,1){60}}
\multiput(0,10)(2,0){40}{\line(0,1){.1}}
\multiput(0,20)(2,0){40}{\line(0,1){.1}}
\multiput(0,30)(2,0){40}{\line(0,1){.1}}
\multiput(0,40)(2,0){40}{\line(0,1){.1}}
\multiput(0,50)(2,0){40}{\line(0,1){.1}}
\multiput(0,60)(2,0){40}{\line(0,1){.1}}
\multiput(10,0)(0,2){30}{\line(1,0){.1}}
\multiput(20,0)(0,2){30}{\line(1,0){.1}}
\multiput(30,0)(0,2){30}{\line(1,0){.1}}
\multiput(40,0)(0,2){30}{\line(1,0){.1}}
\multiput(50,0)(0,2){30}{\line(1,0){.1}}
\multiput(60,0)(0,2){30}{\line(1,0){.1}}
\multiput(70,0)(0,2){30}{\line(1,0){.1}}
  \put(0,0){\tableau{{}}}
\put(20,30){\tableau{{}\\{}\\{}\\{}}} \put(30,10){\tableau{{}\\{}}}
\put(60,10){\tableau{{}\\{}}}
\end{picture}
\end{center}

The following figures illustrate $1^{\searrow} \mathcal{T}$,
$2^{\searrow} \mathcal{T}$ and $3^{\searrow} \mathcal{T}$
respectively.

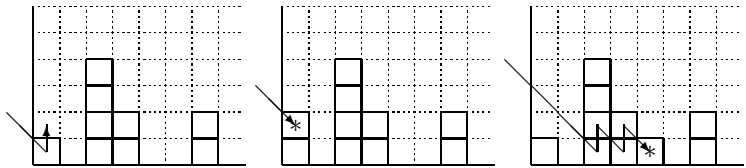
\begin{center}
\begin{picture}(80,60)
\multiput(0,0)(0,0){1}{\line(1,0){80}}
\multiput(0,0)(0,0){1}{\line(0,1){60}}
\multiput(0,10)(2,0){40}{\line(0,1){.1}}
\multiput(0,20)(2,0){40}{\line(0,1){.1}}
\multiput(0,30)(2,0){40}{\line(0,1){.1}}
\multiput(0,40)(2,0){40}{\line(0,1){.1}}
\multiput(0,50)(2,0){40}{\line(0,1){.1}}
\multiput(0,60)(2,0){40}{\line(0,1){.1}}
\multiput(10,0)(0,2){30}{\line(1,0){.1}}
\multiput(20,0)(0,2){30}{\line(1,0){.1}}
\multiput(30,0)(0,2){30}{\line(1,0){.1}}
\multiput(40,0)(0,2){30}{\line(1,0){.1}}
\multiput(50,0)(0,2){30}{\line(1,0){.1}}
\multiput(60,0)(0,2){30}{\line(1,0){.1}}
\multiput(70,0)(0,2){30}{\line(1,0){.1}}
 \put(0,0){\tableau{{}}}
\put(20,30){\tableau{{}\\{}\\{}\\{}}} \put(30,10){\tableau{{}\\{}}}
\put(60,10){\tableau{{}\\{}}}
\multiput(-10,20)(0,0){1}{\line(1,-1){15}}
\multiput(5,5)(0,0){1}{\vector(0,1){10}}
\end{picture}
\hskip.15in
\begin{picture}(80,60)
\multiput(0,0)(0,0){1}{\line(1,0){80}}
\multiput(0,0)(0,0){1}{\line(0,1){60}}
\multiput(0,10)(2,0){40}{\line(0,1){.1}}
\multiput(0,20)(2,0){40}{\line(0,1){.1}}
\multiput(0,30)(2,0){40}{\line(0,1){.1}}
\multiput(0,40)(2,0){40}{\line(0,1){.1}}
\multiput(0,50)(2,0){40}{\line(0,1){.1}}
\multiput(0,60)(2,0){40}{\line(0,1){.1}}
\multiput(10,0)(0,2){30}{\line(1,0){.1}}
\multiput(20,0)(0,2){30}{\line(1,0){.1}}
\multiput(30,0)(0,2){30}{\line(1,0){.1}}
\multiput(40,0)(0,2){30}{\line(1,0){.1}}
\multiput(50,0)(0,2){30}{\line(1,0){.1}}
\multiput(60,0)(0,2){30}{\line(1,0){.1}}
\multiput(70,0)(0,2){30}{\line(1,0){.1}}
  \put(0,0){\tableau{{}}}
\put(20,30){\tableau{{}\\{}\\{}\\{}}} \put(30,10){\tableau{{}\\{}}}
\put(60,10){\tableau{{}\\{}}}
 \put(0,10){\tableau{{\bf*}}}
\multiput(-10,30)(0,0){1}{\vector(1,-1){15}}
\end{picture}
\hskip.15in
\begin{picture}(80,60)
\multiput(0,0)(0,0){1}{\line(1,0){80}}
\multiput(0,0)(0,0){1}{\line(0,1){60}}
\multiput(0,10)(2,0){40}{\line(0,1){.1}}
\multiput(0,20)(2,0){40}{\line(0,1){.1}}
\multiput(0,30)(2,0){40}{\line(0,1){.1}}
\multiput(0,40)(2,0){40}{\line(0,1){.1}}
\multiput(0,50)(2,0){40}{\line(0,1){.1}}
\multiput(0,60)(2,0){40}{\line(0,1){.1}}
\multiput(10,0)(0,2){30}{\line(1,0){.1}}
\multiput(20,0)(0,2){30}{\line(1,0){.1}}
\multiput(30,0)(0,2){30}{\line(1,0){.1}}
\multiput(40,0)(0,2){30}{\line(1,0){.1}}
\multiput(50,0)(0,2){30}{\line(1,0){.1}}
\multiput(60,0)(0,2){30}{\line(1,0){.1}}
\multiput(70,0)(0,2){30}{\line(1,0){.1}}
  \put(0,0){\tableau{{}}}
\put(20,30){\tableau{{}\\{}\\{}\\{}}} \put(30,10){\tableau{{}\\{}}}
\put(60,10){\tableau{{}\\{}}}
 \put(40,0){\tableau{{\bf*}}}
\multiput(-10,40)(0,0){1}{\line(1,-1){35}}
\multiput(25,5)(0,0){1}{\line(0,1){10}}
\multiput(25,15)(0,0){1}{\line(1,-1){10}}
\multiput(35,5)(0,0){1}{\line(0,1){10}}
\multiput(35,15)(0,0){1}{\vector(1,-1){10}}
\end{picture}\hskip.15in
\end{center}

For the sliding  $1^{\searrow} \mathcal{T}$, observe that
$\mathcal{T}$ has no cells   on $x+y=0$. Here
$(1,0)\in\mathcal{T}_1 $ but $(1,1)\not \in\mathcal{T}_1 $,
therefore by \textbf{(S1)}(b), the slide $1^{\searrow} \mathcal{T}$ terminates.

For the sliding $2^{\searrow} \mathcal{T}$, observe that
 $\mathcal{T}_1$ is
the first tower from the left which has a cell on $x+y=1$. Here
$(1,0)\in \mathcal{T}_1$ but $(1,1)\not \in \mathcal{T}_1$.
Therefore by \textbf{(S2)}(a), the sliding $2^{\searrow}
\mathcal{T}$ creates a new cell $(1,1)$ on top of $\mathcal{T}_1$
which is indicated by a star in the above  picture.

For the sliding $3^{\searrow} \mathcal{T}$, observe that
$\mathcal{T}$ has no cells lying on $x+y=2$, and
$(3,0)\in\mathcal{T}_3$ and $(3,1)\in\mathcal{T}_3$. Therefore by
\textbf{(S1)}(c)
$$3^{\searrow} \mathcal{T}=(\mathcal{T}_1,\mathcal{T}_2,\mathcal{T}_3)\sqcup 4^{\searrow}(\mathcal{T}_4,\mathcal{T}_5,\mathcal{T}_6, \mathcal{T}_7)$$
which is indicated by a zigzag line as above. Here the slide
$4^{\searrow}(\mathcal{T}_4,\mathcal{T}_5,\mathcal{T}_6,
\mathcal{T}_7)$ also satisfies \textbf{(S1)}(c),  therefore
$$3^{\searrow} \mathcal{T}=(\mathcal{T}_1,\mathcal{T}_2,\mathcal{T}_3)\sqcup 4^{\searrow}(\mathcal{T}_4,\mathcal{T}_5,\mathcal{T}_6,
\mathcal{T}_7)=(\mathcal{T}_1,\mathcal{T}_2,\mathcal{T}_3,\mathcal{T}_4)\sqcup
5^{\searrow}(\mathcal{T}_5,\mathcal{T}_6, \mathcal{T}_7).$$ Now
$5^{\searrow}(\mathcal{T}_5,\mathcal{T}_6, \mathcal{T}_7)$ satisfies
\textbf{(S1)}(a), therefore sliding $3^{\searrow} \mathcal{T}$
creates a new cell $(5,0)$ at the end which is indicated by a star
above.

The sliding of the numbers $4,5,6,7,8$ and $9$ on to $\mathcal{T}$ is
illustrated by the following diagrams respectively, where the stars
 represent the new cells created by the nonterminating slides. The
 verification is very similar to the one above and is left to the
 reader.

\vskip.10in
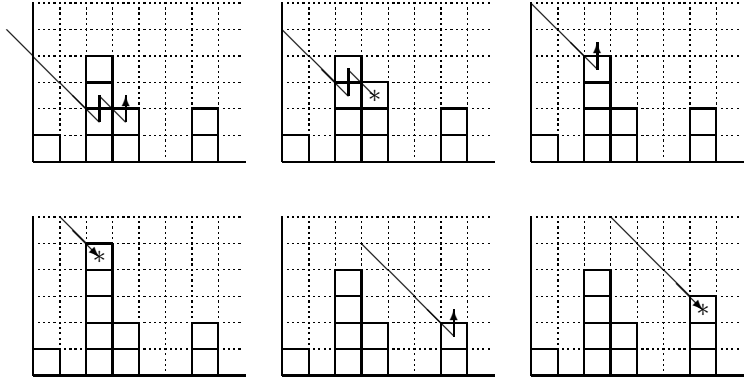
\begin{center}
\begin{picture}(80,60)
\multiput(0,0)(0,0){1}{\line(1,0){80}}
\multiput(0,0)(0,0){1}{\line(0,1){60}}
\multiput(0,10)(2,0){40}{\line(0,1){.1}}
\multiput(0,20)(2,0){40}{\line(0,1){.1}}
\multiput(0,30)(2,0){40}{\line(0,1){.1}}
\multiput(0,40)(2,0){40}{\line(0,1){.1}}
\multiput(0,50)(2,0){40}{\line(0,1){.1}}
\multiput(0,60)(2,0){40}{\line(0,1){.1}}
\multiput(10,0)(0,2){30}{\line(1,0){.1}}
\multiput(20,0)(0,2){30}{\line(1,0){.1}}
\multiput(30,0)(0,2){30}{\line(1,0){.1}}
\multiput(40,0)(0,2){30}{\line(1,0){.1}}
\multiput(50,0)(0,2){30}{\line(1,0){.1}}
\multiput(60,0)(0,2){30}{\line(1,0){.1}}
\multiput(70,0)(0,2){30}{\line(1,0){.1}}
 \put(0,0){\tableau{{}}}
\put(20,30){\tableau{{}\\{}\\{}\\{}}} \put(30,10){\tableau{{}\\{}}}
\put(60,10){\tableau{{}\\{}}}
\multiput(-10,50)(0,0){1}{\line(1,-1){35}}
\multiput(25,15)(0,0){1}{\line(0,1){10}}
\multiput(25,25)(0,0){1}{\line(1,-1){10}}
\multiput(35,15)(0,0){1}{\vector(0,1){10}}
\end{picture}
\hskip.15in
\begin{picture}(80,60)
\multiput(0,0)(0,0){1}{\line(1,0){80}}
\multiput(0,0)(0,0){1}{\line(0,1){60}}
\multiput(0,10)(2,0){40}{\line(0,1){.1}}
\multiput(0,20)(2,0){40}{\line(0,1){.1}}
\multiput(0,30)(2,0){40}{\line(0,1){.1}}
\multiput(0,40)(2,0){40}{\line(0,1){.1}}
\multiput(0,50)(2,0){40}{\line(0,1){.1}}
\multiput(0,60)(2,0){40}{\line(0,1){.1}}
\multiput(10,0)(0,2){30}{\line(1,0){.1}}
\multiput(20,0)(0,2){30}{\line(1,0){.1}}
\multiput(30,0)(0,2){30}{\line(1,0){.1}}
\multiput(40,0)(0,2){30}{\line(1,0){.1}}
\multiput(50,0)(0,2){30}{\line(1,0){.1}}
\multiput(60,0)(0,2){30}{\line(1,0){.1}}
\multiput(70,0)(0,2){30}{\line(1,0){.1}}
  \put(0,0){\tableau{{}}}
\put(20,30){\tableau{{}\\{}\\{}\\{}}} \put(30,10){\tableau{{}\\{}}}
\put(60,10){\tableau{{}\\{}}}
 \put(30,20){\tableau{{\bf*}}}
\multiput(0,50)(0,0){1}{\line(1,-1){25}}
\multiput(25,25)(0,0){1}{\line(0,1){10}}
\multiput(25,35)(0,0){1}{\line(1,-1){10}}
\end{picture}
\hskip.15in
\begin{picture}(80,60)
\multiput(0,0)(0,0){1}{\line(1,0){80}}
\multiput(0,0)(0,0){1}{\line(0,1){60}}
\multiput(0,10)(2,0){40}{\line(0,1){.1}}
\multiput(0,20)(2,0){40}{\line(0,1){.1}}
\multiput(0,30)(2,0){40}{\line(0,1){.1}}
\multiput(0,40)(2,0){40}{\line(0,1){.1}}
\multiput(0,50)(2,0){40}{\line(0,1){.1}}
\multiput(0,60)(2,0){40}{\line(0,1){.1}}
\multiput(10,0)(0,2){30}{\line(1,0){.1}}
\multiput(20,0)(0,2){30}{\line(1,0){.1}}
\multiput(30,0)(0,2){30}{\line(1,0){.1}}
\multiput(40,0)(0,2){30}{\line(1,0){.1}}
\multiput(50,0)(0,2){30}{\line(1,0){.1}}
\multiput(60,0)(0,2){30}{\line(1,0){.1}}
\multiput(70,0)(0,2){30}{\line(1,0){.1}}
 \put(0,0){\tableau{{}}}
\put(20,30){\tableau{{}\\{}\\{}\\{}}} \put(30,10){\tableau{{}\\{}}}
\put(60,10){\tableau{{}\\{}}}
\multiput(0,60)(0,0){1}{\line(1,-1){25}}
\multiput(25,35)(0,0){1}{\vector(0,1){10}}
\end{picture}
\end{center}

 \vskip.10in
\begin{center}
\begin{picture}(80,60)
\multiput(0,0)(0,0){1}{\line(1,0){80}}
\multiput(0,0)(0,0){1}{\line(0,1){60}}
\multiput(0,10)(2,0){40}{\line(0,1){.1}}
\multiput(0,20)(2,0){40}{\line(0,1){.1}}
\multiput(0,30)(2,0){40}{\line(0,1){.1}}
\multiput(0,40)(2,0){40}{\line(0,1){.1}}
\multiput(0,50)(2,0){40}{\line(0,1){.1}}
\multiput(0,60)(2,0){40}{\line(0,1){.1}}
\multiput(10,0)(0,2){30}{\line(1,0){.1}}
\multiput(20,0)(0,2){30}{\line(1,0){.1}}
\multiput(30,0)(0,2){30}{\line(1,0){.1}}
\multiput(40,0)(0,2){30}{\line(1,0){.1}}
\multiput(50,0)(0,2){30}{\line(1,0){.1}}
\multiput(60,0)(0,2){30}{\line(1,0){.1}}
\multiput(70,0)(0,2){30}{\line(1,0){.1}}
  \put(0,0){\tableau{{}}}
\put(20,30){\tableau{{}\\{}\\{}\\{}}} \put(30,10){\tableau{{}\\{}}}
\put(60,10){\tableau{{}\\{}}}
 \put(20,40){\tableau{{\bf*}}}
\multiput(10,60)(0,0){1}{\vector(1,-1){15}} \end{picture}
\hskip.15in
\begin{picture}(80,60)
\multiput(0,0)(0,0){1}{\line(1,0){80}}
\multiput(0,0)(0,0){1}{\line(0,1){60}}
\multiput(0,10)(2,0){40}{\line(0,1){.1}}
\multiput(0,20)(2,0){40}{\line(0,1){.1}}
\multiput(0,30)(2,0){40}{\line(0,1){.1}}
\multiput(0,40)(2,0){40}{\line(0,1){.1}}
\multiput(0,50)(2,0){40}{\line(0,1){.1}}
\multiput(0,60)(2,0){40}{\line(0,1){.1}}
\multiput(10,0)(0,2){30}{\line(1,0){.1}}
\multiput(20,0)(0,2){30}{\line(1,0){.1}}
\multiput(30,0)(0,2){30}{\line(1,0){.1}}
\multiput(40,0)(0,2){30}{\line(1,0){.1}}
\multiput(50,0)(0,2){30}{\line(1,0){.1}}
\multiput(60,0)(0,2){30}{\line(1,0){.1}}
\multiput(70,0)(0,2){30}{\line(1,0){.1}}
 \put(0,0){\tableau{{}}}
\put(20,30){\tableau{{}\\{}\\{}\\{}}} \put(30,10){\tableau{{}\\{}}}
\put(60,10){\tableau{{}\\{}}}
\multiput(30,50)(0,0){1}{\line(1,-1){35}}
\multiput(65,15)(0,0){1}{\vector(0,1){10}}
\end{picture}
\hskip.15in
\begin{picture}(80,60)
\multiput(0,0)(0,0){1}{\line(1,0){80}}
\multiput(0,0)(0,0){1}{\line(0,1){60}}
\multiput(0,10)(2,0){40}{\line(0,1){.1}}
\multiput(0,20)(2,0){40}{\line(0,1){.1}}
\multiput(0,30)(2,0){40}{\line(0,1){.1}}
\multiput(0,40)(2,0){40}{\line(0,1){.1}}
\multiput(0,50)(2,0){40}{\line(0,1){.1}}
\multiput(0,60)(2,0){40}{\line(0,1){.1}}
\multiput(10,0)(0,2){30}{\line(1,0){.1}}
\multiput(20,0)(0,2){30}{\line(1,0){.1}}
\multiput(30,0)(0,2){30}{\line(1,0){.1}}
\multiput(40,0)(0,2){30}{\line(1,0){.1}}
\multiput(50,0)(0,2){30}{\line(1,0){.1}}
\multiput(60,0)(0,2){30}{\line(1,0){.1}}
\multiput(70,0)(0,2){30}{\line(1,0){.1}}
  \put(0,0){\tableau{{}}}
\put(20,30){\tableau{{}\\{}\\{}\\{}}} \put(30,10){\tableau{{}\\{}}}
\put(60,10){\tableau{{}\\{}}}
 \put(60,20){\tableau{{\bf*}}}
\multiput(30,60)(0,0){1}{\vector(1,-1){35}}
\end{picture}
\end{center}
\end{ex}

In the following we make precise the relation between sliding of a number and flight number of
a cell by proving that they are mutually inverse processes.
The proof of this result relies on the observation that the flight path of a cell
is travelled in the reverse direction by the slide of the flight number.  We give below the detailed proof of the first part and leave the proof of the second part.

\begin{lem}\label{lem:flightVSslide}
Let $\mathcal T$ be a tower diagram.
\begin{enumerate}
\item[(a)] If $\alpha$ is a positive integer such that
$\alpha^\searrow\mathcal T= \mathcal T \cup \{d\}$ for some cell $d$, then $d$ is a corner cell in $\alpha^\searrow\mathcal T$ and moreover
$$
\mathrm{flight}\#(d,\alpha^{\searrow}\mathcal T) =\alpha.$$
\item[(b)] If $c$ is a corner cell with flight number $\beta$, then
$$\beta^\searrow(c^\nwarrow\mathcal T) = \mathcal T.$$
\end{enumerate}
\end{lem}

\begin{proof} (a)  We consider the case that   $d=(i,j)$  is  created by  $\alpha^{\searrow}
\mathcal{T}$ subject to  either \textbf{(S2)}(a) or \textbf{(S2)}(c)  of
Definition~\ref{def:sliding} .

If $d=(i,j)$ is created subject to  \textbf{(S2)}(a) then $d=(i,\alpha-i)$ and no tower to the left of $\mathcal{T}_i$ contains a cell on the diagonal
$x+y=\alpha-1$. Therefore $d$ in $\alpha^{\searrow}
\mathcal{T}$ satisfies  \textbf{(F1)} of
Definition~\ref{def:flight} i.e., it is a top cell having a flightpath which consists only of itself. Hence its flight number in this tableau is also $\alpha$.

Now if $(i,j)$ is the cell created by $\alpha^{\searrow}
\mathcal{T}$ subject to  \textbf{(S2)}(c) of
Definition~\ref{def:sliding}, then no tower to the left of $\mathcal{T}_i$ contains a
cell  on the diagonal  $x+y=\alpha-1$. Furthermore  $(i,\alpha-i)\in
\mathcal{T}_i$,  $(i,\alpha-i+1) \in \mathcal{T}_i$ and
$$\alpha^{\searrow} \mathcal{T}:=
(\mathcal{T}_1,\ldots,\mathcal{T}_i) \sqcup (\alpha+1)^{\searrow}
(\mathcal{T}_{i+1},\ldots,\mathcal{T}_k) $$
In this case,  $d$ lies in $(\alpha+1)^{\searrow}
(\mathcal{T}_{i+1},\ldots,\mathcal{T}_k)$ and the size of
$(\mathcal{T}_{i+1},\ldots,\mathcal{T}_k)$ is strictly less than the
size of $\mathcal{T}$.

By induction on the size of the tower
diagram, we may assume that $d$ is a corner cell of
$(\alpha+1)^{\searrow} (\mathcal{T}_{i+1},\ldots,\mathcal{T}_k)$ and
that its flight number in this tableau is $(\alpha+1)$.
Therefore
the flight path of $d$ in  $(\alpha+1)^{\searrow} (\mathcal{T}_{i+1},\ldots,\mathcal{T}_k)$ contains some  cells lying on
$x+y=\alpha+1$ and  let  $(i',j')\in \mathcal{T}_{i'} $ be the
lexicographically first among all such cells in this tableau. Hence no towers between $\mathcal{T}_{i} $ and  $\mathcal{T}_{i'} $ has a cell lying
on $x+y=\alpha+1$.   Observe that
$$\begin{aligned}
 \mathrm{flightpath}(d,\mathcal
T) =& ~\mathrm{flightpath}(d,(\alpha+1)^{\searrow}
(\mathcal{T}_{i+1},\ldots,\mathcal{T}_k))\\ & \cup
\mathrm{flightpath}((i',j'),\mathcal T) \end{aligned}$$
and that $(i,\alpha-i+1)$ of   $\mathcal{T}_i$ and $(i',j')\in \mathcal{T}_{i'} $ lies in the same diagonal $x+y=\alpha+1$. Therefore by \textbf{(F2)} of
Definition~\ref{def:flight}
$$\mathrm{flightpath}((i',j'),\mathcal T)=\{(i',j'),(i,\alpha-i+1)\}\cup\mathrm{flightpath}((i,\alpha-i),\mathcal T).$$
 On the other hand  $\mathrm{flightpath}((i,\alpha-i),\mathcal T) =\{(i,\alpha-i)\}$ by \textbf{(F1)} of
Definition~\ref{def:flight}. Now  $(i,\alpha-i)$ is the lexicographically first cell in $\mathrm{flightpath}(d,\mathcal
T)$. Therefore the flight number of $d$ in $\mathcal
T$ is $\alpha $ as desired.

We will skip the case defined by \textbf{(S1)} (a) and (b), since   the related analysis is very similar.
\end{proof}

Our next aim is to describe the relations between consecutive
slides. We state these result as two separate lemmas. Because the
proofs of the lemmas are technical and very long, we include them in
\ref{Appendix}.

The first lemma examines the case where the integers that are to be slided are far away from each other.

\begin{lem}\label{Lemma:SlidingFarAway}
Let $\mathcal{T}=(\mathcal{T}_1,\ldots,\mathcal{T}_k)$ be a tower diagram and $\alpha$ and $\beta$ be
positive integers satisfying $|{\beta}-{\alpha}| \geq 2$.
Then either the equality $$\alpha^{\searrow}(\beta^{\searrow}\mathcal{T})=
\beta^{\searrow} (\alpha^{\searrow}\mathcal{T})$$ holds or both slides
terminate.
\end{lem}
The remaining case is where the numbers that are to be slided are close to each other
is given in the following result.

\begin{lem}\label{Lemma:SlidingClose}
Let $\mathcal{T}=(\mathcal{T}_1,\mathcal{T}_2,\ldots)$ be a tower diagram and $\alpha$  be a positive integer. Then either the equality
$$\alpha^{\searrow}((\alpha+1)^{\searrow}(\alpha
^{\searrow}\mathcal{T}))= (\alpha+1)^{\searrow}(\alpha^{\searrow}
((\alpha+1)^{\searrow}\mathcal{T}))$$
holds or both slides terminate.
\end{lem}

Now we introduce a sliding and recording (SR) algorithm (sliding algorithm, for short)
on the set of finite words on $\mathbb{Z}^+$ which produces a pair of
tower tableaux of the same shape, whenever the algorithm does  not
terminate. We shall prove that one of these tableaux is canonically determined by the
shape of the tableau whereas the other is standard. In particular, we shall obtain a
criterion on words to be in the image of the reading function
defined in Section \ref{Section:towerdiagrams}.

\begin{defn} Let
$\alpha=\alpha_1\alpha_2\ldots\alpha_n$ be a word on $\mathbb{Z}^+$.
Then SR (\textbf{\textit{sliding}} and \textbf{\textit{recording}}) algorithm on $\alpha$ produces, if it does not terminate, two tower tableaux of the same shape, the \textbf{\textit{sliding tableau}} $S(\alpha)$ and the \textbf{\textit{recording tableau}} $R(\alpha)$. These tableaux are obtained through a sequence of pairs of the \textit{same shape} tower tableaux
$$ (S_1,R_1),(S_2,R_2)\ldots, (S_n,R_n)=(S(\alpha),R(\alpha))$$
where $S_1=\{[(\alpha_1,0),\alpha_1]\}$ and
$R_1=\{[(\alpha_1,0),1]\}$, and for $1<k\leq n$, $S_k$ (and $R_k$)
is obtained by sliding $\alpha_k$ over  $S_{k-1}$ (and respectively
$R_{k-1}$) by the following rule:

Let $\mathcal{T}_{k-1}:=\mathrm{shape}(S_{k-1})=\mathrm{shape}(R_{k-1})$.
If $\alpha^{\searrow} \mathcal{T}_{k-1}:= \mathcal{T}_{k-1} \cup \{ (i,j)\}$ then we
put
$$ \begin{aligned}S_k:=& \alpha^{\searrow} S_{k-1}=S_{k-1}\cup \{ [(i,j),i+j]\} \\
 R_k:=&\alpha^{\searrow} R_{k-1}= R_{k-1}\cup \{ [(i,j),k]\}
\end{aligned}.$$
Otherwise SR algorithm terminates without a result.

We denote the set consisting of all words on $\mathbb{Z}^+$ on which
$\mathrm{SR}$ algorithm does not terminate by
$$\mathrm{SRW}(\mathbb{Z}^+).$$
\end{defn}

\begin{ex} Sliding and recording algorithm applied on  the word $\alpha=784534561$
gives the following tower tableaux $S(\alpha)$ and $R(\alpha)$, respectively.
\begin{center}
\begin{picture}(80,60)
\put(-40,30){$S(\alpha) = $}
\multiput(0,0)(0,0){1}{\line(1,0){80}}
\multiput(0,0)(0,0){1}{\line(0,1){60}}
\multiput(0,10)(2,0){40}{\line(0,1){.1}}
\multiput(0,20)(2,0){40}{\line(0,1){.1}}
\multiput(0,30)(2,0){40}{\line(0,1){.1}}
\multiput(0,40)(2,0){40}{\line(0,1){.1}}
\multiput(0,50)(2,0){40}{\line(0,1){.1}}
\multiput(0,60)(2,0){40}{\line(0,1){.1}}
\multiput(10,0)(0,2){30}{\line(1,0){.1}}
\multiput(20,0)(0,2){30}{\line(1,0){.1}}
\multiput(30,0)(0,2){30}{\line(1,0){.1}}
\multiput(40,0)(0,2){30}{\line(1,0){.1}}
\multiput(50,0)(0,2){30}{\line(1,0){.1}}
\multiput(60,0)(0,2){30}{\line(1,0){.1}}
\multiput(70,0)(0,2){30}{\line(1,0){.1}}
  \put(0,0){\tableau{{1}}}
\put(20,30){\tableau{{6}\\{5}\\{4}\\{3}}}
\put(30,10){\tableau{{5}\\{4}}} \put(60,10){\tableau{{8}\\{7}}}
\end{picture}\hskip.7in
\begin{picture}(80,60)
\put(-40,30){$R(\alpha)=$} \multiput(0,0)(0,0){1}{\line(1,0){80}}
\multiput(0,0)(0,0){1}{\line(0,1){60}}
\multiput(0,10)(2,0){40}{\line(0,1){.1}}
\multiput(0,20)(2,0){40}{\line(0,1){.1}}
\multiput(0,30)(2,0){40}{\line(0,1){.1}}
\multiput(0,40)(2,0){40}{\line(0,1){.1}}
\multiput(0,50)(2,0){40}{\line(0,1){.1}}
\multiput(0,60)(2,0){40}{\line(0,1){.1}}
\multiput(10,0)(0,2){30}{\line(1,0){.1}}
\multiput(20,0)(0,2){30}{\line(1,0){.1}}
\multiput(30,0)(0,2){30}{\line(1,0){.1}}
\multiput(40,0)(0,2){30}{\line(1,0){.1}}
\multiput(50,0)(0,2){30}{\line(1,0){.1}}
\multiput(60,0)(0,2){30}{\line(1,0){.1}}
\multiput(70,0)(0,2){30}{\line(1,0){.1}}
  \put(0,0){\tableau{{9}}}
\put(20,30){\tableau{{8}\\{7}\\{6}\\{5}}}
\put(30,10){\tableau{{4}\\{3}}} \put(60,10){\tableau{{2}\\{1}}}
\end{picture}
\end{center}

Let $\mathcal{T}=\mathrm{shape}(S(\alpha))=\mathrm{shape}(R(\alpha))$. Recall from the previous example that
the only numbers whose sliding into
$\mathcal{T}$ do not terminates are $2$, $3$, $5$, $7$, $9$ and any integer $>9$.  Below we illustrate the sliding
of these numbers into  $S(\alpha)$ respectively. Note that for an integer $k\geq 9$ the sliding of $k$ into $S(\alpha)
$ is
obtained by adding $[(k,0),k]$ to $S$. For finding the recording tableaux, on the other hand, one only need to label the new cell in each tableaux  by $10$.

\vskip.1in
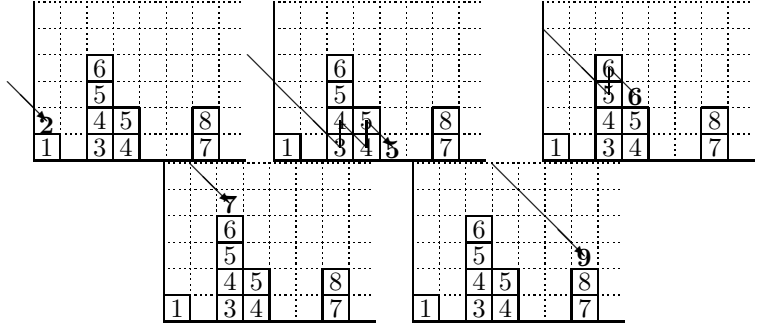
\begin{center}
\begin{picture}(80,60)
\multiput(0,0)(0,0){1}{\line(1,0){80}}
\multiput(0,0)(0,0){1}{\line(0,1){60}}
\multiput(0,10)(2,0){40}{\line(0,1){.1}}
\multiput(0,20)(2,0){40}{\line(0,1){.1}}
\multiput(0,30)(2,0){40}{\line(0,1){.1}}

\multiput(0,40)(2,0){40}{\line(0,1){.1}}
\multiput(0,50)(2,0){40}{\line(0,1){.1}}
\multiput(0,60)(2,0){40}{\line(0,1){.1}}
\multiput(10,0)(0,2){30}{\line(1,0){.1}}
\multiput(20,0)(0,2){30}{\line(1,0){.1}}
\multiput(30,0)(0,2){30}{\line(1,0){.1}}
\multiput(40,0)(0,2){30}{\line(1,0){.1}}
\multiput(50,0)(0,2){30}{\line(1,0){.1}}
\multiput(60,0)(0,2){30}{\line(1,0){.1}}
\multiput(70,0)(0,2){30}{\line(1,0){.1}}
 \put(0,0){\tableau{{1}}}
\put(20,30){\tableau{{6}\\{5}\\{4}\\{3}}}
\put(30,10){\tableau{{5}\\{4}}} \put(60,10){\tableau{{8}\\{7}}}
 \put(2,10){\bf2}
\multiput(-10,30)(0,0){1}{\vector(1,-1){15}}
\end{picture}\hskip.15in
\begin{picture}(80,60)
\multiput(0,0)(0,0){1}{\line(1,0){80}}
\multiput(0,0)(0,0){1}{\line(0,1){60}}
\multiput(0,10)(2,0){40}{\line(0,1){.1}}
\multiput(0,20)(2,0){40}{\line(0,1){.1}}
\multiput(0,30)(2,0){40}{\line(0,1){.1}}
\multiput(0,40)(2,0){40}{\line(0,1){.1}}
\multiput(0,50)(2,0){40}{\line(0,1){.1}}
\multiput(0,60)(2,0){40}{\line(0,1){.1}}
\multiput(10,0)(0,2){30}{\line(1,0){.1}}
\multiput(20,0)(0,2){30}{\line(1,0){.1}}
\multiput(30,0)(0,2){30}{\line(1,0){.1}}
\multiput(40,0)(0,2){30}{\line(1,0){.1}}
\multiput(50,0)(0,2){30}{\line(1,0){.1}}
\multiput(60,0)(0,2){30}{\line(1,0){.1}}
\multiput(70,0)(0,2){30}{\line(1,0){.1}}
 \put(0,0){\tableau{{1}}}
\put(20,30){\tableau{{6}\\{5}\\{4}\\{3}}}
\put(30,10){\tableau{{5}\\{4}}} \put(60,10){\tableau{{8}\\{7}}}
 \put(42,1){\bf5}
\multiput(-10,40)(0,0){1}{\line(1,-1){35}}
\multiput(25,5)(0,0){1}{\line(0,1){10}}
\multiput(25,15)(0,0){1}{\line(1,-1){10}}
\multiput(35,5)(0,0){1}{\line(0,1){10}}
\multiput(35,15)(0,0){1}{\vector(1,-1){10}}
\end{picture}\hskip.15in
\hskip.15in
\begin{picture}(80,60)
\multiput(0,0)(0,0){1}{\line(1,0){80}}
\multiput(0,0)(0,0){1}{\line(0,1){60}}
\multiput(0,10)(2,0){40}{\line(0,1){.1}}
\multiput(0,20)(2,0){40}{\line(0,1){.1}}
\multiput(0,30)(2,0){40}{\line(0,1){.1}}
\multiput(0,40)(2,0){40}{\line(0,1){.1}}
\multiput(0,50)(2,0){40}{\line(0,1){.1}}
\multiput(0,60)(2,0){40}{\line(0,1){.1}}
\multiput(10,0)(0,2){30}{\line(1,0){.1}}
\multiput(20,0)(0,2){30}{\line(1,0){.1}}
\multiput(30,0)(0,2){30}{\line(1,0){.1}}
\multiput(40,0)(0,2){30}{\line(1,0){.1}}
\multiput(50,0)(0,2){30}{\line(1,0){.1}}
\multiput(60,0)(0,2){30}{\line(1,0){.1}}
\multiput(70,0)(0,2){30}{\line(1,0){.1}}
 \put(0,0){\tableau{{1}}}
\put(20,30){\tableau{{6}\\{5}\\{4}\\{3}}}
\put(30,10){\tableau{{5}\\{4}}} \put(60,10){\tableau{{8}\\{7}}}
\put(32,21){\bf{{6}}} \multiput(0,50)(0,0){1}{\line(1,-1){25}}
\multiput(25,25)(0,0){1}{\line(0,1){10}}
\multiput(25,35)(0,0){1}{\line(1,-1){10}}
\end{picture}\hskip.15in
\begin{picture}(80,60)
\multiput(0,0)(0,0){1}{\line(1,0){80}}
\multiput(0,0)(0,0){1}{\line(0,1){60}}
\multiput(0,10)(2,0){40}{\line(0,1){.1}}
\multiput(0,20)(2,0){40}{\line(0,1){.1}}
\multiput(0,30)(2,0){40}{\line(0,1){.1}}
\multiput(0,40)(2,0){40}{\line(0,1){.1}}
\multiput(0,50)(2,0){40}{\line(0,1){.1}}
\multiput(0,60)(2,0){40}{\line(0,1){.1}}
\multiput(10,0)(0,2){30}{\line(1,0){.1}}
\multiput(20,0)(0,2){30}{\line(1,0){.1}}
\multiput(30,0)(0,2){30}{\line(1,0){.1}}
\multiput(40,0)(0,2){30}{\line(1,0){.1}}
\multiput(50,0)(0,2){30}{\line(1,0){.1}}
\multiput(60,0)(0,2){30}{\line(1,0){.1}}
\multiput(70,0)(0,2){30}{\line(1,0){.1}}
 \put(0,0){\tableau{{1}}}
\put(20,30){\tableau{{6}\\{5}\\{4}\\{3}}}
\put(30,10){\tableau{{5}\\{4}}} \put(60,10){\tableau{{8}\\{7}}}
 \put(22,41){\bf{{7}}}
\multiput(10,60)(0,0){1}{\vector(1,-1){15}} \end{picture}
\hskip.15in
\begin{picture}(80,60)
\multiput(0,0)(0,0){1}{\line(1,0){80}}
\multiput(0,0)(0,0){1}{\line(0,1){60}}
\multiput(0,10)(2,0){40}{\line(0,1){.1}}
\multiput(0,20)(2,0){40}{\line(0,1){.1}}
\multiput(0,30)(2,0){40}{\line(0,1){.1}}
\multiput(0,40)(2,0){40}{\line(0,1){.1}}
\multiput(0,50)(2,0){40}{\line(0,1){.1}}
\multiput(0,60)(2,0){40}{\line(0,1){.1}}
\multiput(10,0)(0,2){30}{\line(1,0){.1}}
\multiput(20,0)(0,2){30}{\line(1,0){.1}}
\multiput(30,0)(0,2){30}{\line(1,0){.1}}
\multiput(40,0)(0,2){30}{\line(1,0){.1}}
\multiput(50,0)(0,2){30}{\line(1,0){.1}}
\multiput(60,0)(0,2){30}{\line(1,0){.1}}
\multiput(70,0)(0,2){30}{\line(1,0){.1}}
 \put(0,0){\tableau{{1}}}
\put(20,30){\tableau{{6}\\{5}\\{4}\\{3}}}
\put(30,10){\tableau{{5}\\{4}}} \put(60,10){\tableau{{8}\\{7}}}
 \put(62,21){\bf{{9}}}
\multiput(30,60)(0,0){1}{\vector(1,-1){35}}
\end{picture}
\end{center}
\vskip.15in

\end{ex}

Our next aim is to classify the sliding and recording tableaux. The classification of the sliding tableaux is easy.
\begin{lem} For every   $\alpha\in \mathrm{SRW(\mathbb{Z}^+)}$, the sliding tableau  $S(\alpha)$ is canonically determined by its shape.
Moreover  if $\alpha'$ is another word in
$\mathrm{SRW(\mathbb{Z}^+)}$ satisfying
$\mathrm{shape}(S(\alpha))=\mathrm{shape}(S(\alpha'))$ then $\alpha$
and $\alpha'$ has the same size and $S(\alpha)=S(\alpha')$.
\end{lem}
\begin{proof} The first statement follows from the definition of the sliding tableaux directly. Indeed, it is clear that the cell $(i,j)$ in $S(\alpha)$ has label $i+j$.
For the second statement it is clear that  two tableaux  have the
same number of cells and so do the two words by the sliding algorithm.
The other statement follows from the fact that for any  tower diagram
the associated canonical tower tableau is unique.
\end{proof}
As the lemma shows, the sliding tableau is uniquely determined by
the shape of the corresponding recording tableau. Hence after this point, we
will forget the sliding tableau and concentrate only on the
recording tableau.

Regarding the recording tableaux, the classification is done by
standard labelling as follows.

\begin{lem} Let  $\alpha=\alpha_1\ldots\alpha_n$ be a word in $\mathrm{SRW(\mathbb{Z}^+)}$. Then $R(\alpha)$ is a standard tower tableau.
\end{lem}

\begin{proof} We will proceed by induction on the size of $\alpha$.
For any word of size one, the recording tableau consists of only
one labelled cell $[(i,0),1]$ for some positive integer $i$, and
clearly it is a standard tower tableau. Suppose that the recording
tableau for any word of size $\leq n-1$ in
$\mathrm{SRW(\mathbb{Z}^+)}$ is a standard tower tableau. Let
$R(\alpha)$ be the recording tableau of
$\alpha=\alpha_1\ldots\alpha_n \in \mathrm{SRW(\mathbb{Z}^+)}$.
Then $\alpha_1\ldots\alpha_{n-1}$ is also in
$\mathrm{SRW(\mathbb{Z}^+)}$ and
$$R(\alpha)=R(\alpha_1\ldots\alpha_{n-1}) \cup [(i,j),n]
$$
for some cell $(i,j)$. Here $R(\alpha_1\ldots\alpha_{n-1})$ is a
standard tower tableau by induction and moreover
$$R(\alpha_1\ldots\alpha_{n-1})=R(\alpha)_{\leq n-1}.$$
Now  the definition of standard tower tableau asserts that for
each $1<k \leq n-1$, the cell labelled by $k$ is a corner cell
in the diagram $R(\alpha_1\ldots\alpha_{n-1})_{\leq
k}= R(\alpha_1\ldots\alpha_{k})$ and hence in $R(\alpha)_{\leq k}$.
So we only  need to show that the cell $[(i,j),n]$  is a corner
cell of  $R(\alpha)_{\leq n}=R(\alpha)$.

Let $\mathrm{shape}(R(\alpha)_{< n})=\mathcal T$. Then
 $$\mathrm{shape}(R(\alpha))=\alpha_n\searrow \mathcal T=\mathcal T \cup \{(i,j)\}.$$
where $(i,j)$ is a corner cell in $\alpha_n\searrow \mathcal T$ by Lemma~\ref{lem:flightVSslide}(a) . Hence
 $[(i,j),n]$  is a corner
cell of  $R(\alpha)$ as required.
\end{proof}

By the above results, we have obtained a function
\[
R:\mathrm{SRW(\mathbb{Z}^+)}\rightarrow \mathrm{STT}
\]
from the set $\mathrm{SRW(\mathbb{Z}^+)}$ of words on which the
sliding algorithm does not terminate to the set $\mathrm{STT}$ of
all standard tower tableaux, defined by sending a word $\alpha$ to
the corresponding standard tower tableaux $R(\alpha)$.

Recall that we have also defined the reading function
\[
\mathrm{Read}:\mathrm{STT} \rightarrow \mathrm{W(\mathbb{Z}^+)}.
\]
By Lemma \ref{lem:flightVSslide}, it is celar that the compositions
$R\circ\mathrm{Read}$ and
$\mathrm{Read}\circ R$ are identity on the sets $\mathrm{STT}$ and
$\mathrm{SRW(\mathbb{Z}^+)}$, respectively. As a result, we obtain the following
theorem.
\begin{thm}\label{thm:srw-stt}
There is a bijective correspondence between
\begin{enumerate}
\item[i)] the set $\mathrm{SRW(\mathbb{Z}^+)}$ of words on which the sliding algorithm does not terminate and
\item[ii)] the set $\mathrm{STT}$ of all standard tower tableaux of all shapes
\end{enumerate}
given by $\alpha\mapsto R(\alpha)$ and $R\mapsto \mathrm{Read}(R)$.
\end{thm}

\section{Recording tableaux and reduced words}
\label{Section:recordAndReduced}
Our main result on standard tower tableaux is that they parametrize the
reduced decompositions of permutations. Finally we are ready to
prove this result. More precisely, we prove that the set
$\mathrm{SRW}(\mathbb Z^+)$ of words on which the sliding
algorithm does not terminate coincide with the reduced words for
permutations, see Theorem \ref{thm:SRW2RW}.

Note further that, the set $\mathrm{STT}$ of standard tower
tableaux of all shapes has a natural partition according to
shapes. On the other hand, the set $\mathrm{SRW}(\mathbb Z^+)$ has
a natural partition into classes according to the corresponding
permutation, as described below. The other main result of this
section is that the functions defined in Theorem \ref{thm:srw-stt} preserve these
partitions.

Combining these two results, we conclude that a tower diagram determines a unique
permutation. Vice versa the set of all reduced decompositions of a permutation
is in bijection with the set of all standard tower tableaux of the
unique shape determining the permutation we start with.

Now we argue to prove the above results.
The symmetric group $S_n$ is generated by the set of all
adjacent transpositions
$$S:=\{s_i=(i,i+1) \mid 1\leq i \leq n-1\}$$
subject to the following \textbf{\textit{Coxeter {\mbox{\rm (or}} braid{\mbox{\rm )}}
relations}}:
\begin{enumerate}
\item[i)] $s_is_j=s_js_i ~~\textrm{if}~~ \mid i-j\mid \geq 2$
\item[ii)] $s_is_{i+1}s_i=s_{i+1}s_{i}s_{i+1}$
\item[iii)] $s_i s_i=1$
\end{enumerate}
where $1$ represents the identity permutation. For any $w \in
S_n$, an expression $w=s_{i_1}\ldots s_{i_k}$ is called a \textbf{\textit{word}}
representing $w$. The \textbf{\textit{length}} of the permutation $w$, denoted by $l(w)$, is the minimum number of transpositions in a word representing $w$.

Now if $w=s_{i_1}\ldots s_{i_k}$ and $l(w)=k$ then $s_{i_1}\ldots
s_{i_k}$ is said to be a \textbf{\emph{reduced expression}} or a \textbf{\emph{reduced word}} for $w$.

To consider all finite permutations at once, we write $$\varinjlim_{n}S_n$$ for the direct limit of the groups $S_n$ over all $n$.

With this notation, there is a function
\[
s_{[-]}:\mathrm{W}(\mathbb Z^+)\rightarrow \varinjlim_n S_n
\]
given by sending any word $\alpha= \alpha_1\alpha_2\ldots\alpha_n$ over
$\mathbb{Z}^+$ to the permutation represented by the word
$$s_{[\alpha]}:= s_{\alpha_1\alpha_2\ldots\alpha_n}=s_{\alpha_1}s_{\alpha_2}\ldots s_{\alpha_n}.
$$

Our next aim is to relate this function with SR algorithm and the reading
function. First, we show that whenever the SR algorithm does not terminate on a pair
of words $\alpha$ and $\beta$, then the recording tableaux of $\alpha$
and $\beta$ have the same shape if and only if they correspond to the same
permutation under $s_{[?]}$. We state and prove this result as two separate theorems.
We first have the following part.

\begin{thm}\label{thm:braid2shape} Let $\alpha=\alpha_1 \ldots \alpha_n$ and
$\beta=\beta_1 \ldots \beta_n$ be two words
 in $\mathrm{SRW}(\mathbb{Z}^+)$. If $s_{[\alpha]}=s_{[\beta]}$
then $\mathrm{shape}(R(\alpha))=\mathrm{shape}(R(\beta))$.
\end{thm}

\begin{proof} We proceed by induction on the size of the words. One can easily prove the hypothesis for the words of size $\leq 3$. So assume that hypothesis is true for the words of size $\leq n-1$.

Since $s_{[\alpha]}=s_{[\beta]}$, the words
$s_{\alpha_1}s_{\alpha_2}\ldots s_{\alpha_n}$ and
$s_{\beta_1}s_{\beta_2}\ldots s_{\beta_n}$ must be related by a sequence
of braid relations. On the other hand, in order to prove the claim, it is enough to consider the case where they are related by only one braid relation.

Now if $\alpha_n=\beta _n$ then we still have that the words
$s_{\alpha_1}s_{\alpha_2}\ldots s_{\alpha_{n-1}}$ and
$s_{\beta_1}s_{\beta_2}\ldots s_{\beta_{n-1}}$ are braid related and
by induction,
$\mathrm{shape}(R(\alpha_1\ldots\alpha_{n-1}))= \mathrm{shape}(R(\beta_1\ldots\beta_{n-1}))$. We denote the common shape by $\mathcal{T}$.

On the other hand, we have the equalities
$$R(\alpha)=\alpha_n^{\searrow}R( \alpha_{1}\ldots \alpha_{n-1})
\text{ and } R(\beta)=\alpha_n^{\searrow}R( \beta_{1}\ldots
\beta_{n-1}).$$ Moreover these tableaux are determined by $\mathcal{T}$ and by the
number $\alpha_n$. Therefore the shapes of
$\mathrm{shape}(R(\beta))$ and $\mathrm{shape}(R(\alpha))$ are the same.

Next we assume that $\alpha_n\not =\beta_n$. Then we have two cases:

\textit{Case 1.} The words $s_{[\alpha]}$ and $s_{[\beta]}$ are related by a single
relation of the first type, that is, for some $1\leq i,j \leq n-1$ satisfying
$\mid {i}-{j}\mid \geq 2$, we have $s_{\alpha_{n-1}}=s_i$,
$s_{\alpha_{n}}=s_j$ and
$$\begin{aligned}
s_{[\alpha]}=&s_{\alpha_1 }\ldots s_{\alpha_{n-2}} s_{i}
s_{j} \\
s_{[\beta]}=&s_{\alpha_1 }\ldots s_{\alpha_{n-2}} s_{j} s_{i}.
\end{aligned}
$$

Let $\mathcal{T}= \mathrm{shape}(R( \alpha_{1}\ldots \alpha_{n-2}))$.
Observe that in order to prove our claim, it is enough to show that the equality
$$j^{\searrow}i^{\searrow}\mathcal{T}= i^{\searrow}
j^{\searrow}\mathcal{T}$$ holds whenever $\mid {i}-{j}\mid \geq 2$.
But this follows directly from Lemma \ref{Lemma:SlidingFarAway}.

\textit{Case 2.} The words $s_{[\alpha]}$ and $s_{[\beta]}$ are related by a single
relation of the second type, that is, for some $1\leq i \leq n-1$,
$s_{\alpha_{n-2}}=s_{\alpha_{n}}=s_i$, $s_{\alpha_{n-1}}=s_{i+1}$
and
$$\begin{aligned}
s_{[\alpha]}=&s_{\alpha_1 }\ldots s_{\alpha_{n-3}} s_{i} s_{i+1}
s_{i} \\
s_{[\beta]}=&s_{\alpha_1 }\ldots \ldots s_{\alpha_{n-3}} s_{i+1} s_{i}
s_{i+1}.
\end{aligned}
$$
Let $\mathcal{T}= \mathrm{shape}(R( \alpha_{1}\ldots \alpha_{n-3}))$.
Observe that in order to prove our claim, it is enough to show that the equality
$$i^{\searrow}(i+1)^{\searrow}i ^{\searrow}\mathcal{T}= (i+1)^{\searrow}i^{\searrow} (i+1)^{\searrow}\mathcal{T}$$ holds.
But this follows from Lemma \ref{Lemma:SlidingClose}.
\end{proof}

The next theorem provides the converse for the above theorem.

\begin{thm}\label{thm:shape2braid} Let $\alpha=\alpha_1 \ldots \alpha_n$ and
$\beta=\beta_1 \ldots \beta_n$ be two words
 in $\mathrm{SRW}(\mathbb{Z}^+)$. If $\mathrm{shape}(R(\alpha))=\mathrm{shape}(R(\beta))$
then $s_{[\alpha]}=s_{[\beta]}$.
\end{thm}

\begin{proof}
We argue by induction on the number of cells in $R(\alpha)$. Let $$\alpha = \alpha_1\alpha_2\ldots \alpha_n, \,\,\, \beta=\beta_1\beta_2\ldots \beta_n$$
and assume the result for all words of length less than $n$.

Let $c_\alpha$ (resp. $c_\beta$) be the cell in $R(\alpha)$ (resp. $R(\beta)$) with label $n$. Note that by this choice, the cells $c_\alpha$ and $c_\beta$ are corner cells in $\mathcal T:=$shape$(R(\alpha))$.

There are several cases to consider. The easy case is when $c_\alpha$ and $c_\beta$ are the same cells in $
\mathcal T$. In this case, by Lemma \ref{lem:flightVSslide} (a), we have the equalities
\[
\alpha_n = \mathrm{flight}\#(c_\alpha, \mathcal T) = \mathrm{flight}\#(c_\beta,
\mathcal T) = \beta_n
\]
and also since the shapes of $R(\alpha_1\ldots\alpha_{n-1})$ and  $R(\beta_1\ldots\beta_{n-1})$ coincide, the induction hypothesis implies that
$s_{\alpha_1\ldots\alpha_{n-1}} = s_{\beta_1\ldots\beta_{n-1}}$. Hence $s_{[\alpha]} = s_{[\beta]}$, as required.

For the rest of the proof, we assume that the cells $c_\alpha$ and $c_\beta$ are different and that the cell $c_\alpha$ is on the left.

Let $\alpha^\prime = \alpha_1\ldots\alpha_{n-1}$. Then we have
$R(\alpha^\prime) = c_\alpha^\nwarrow(R(\alpha))$. There remains two cases to
consider.

First, the cell $c_\beta$ can be a corner cell in shape$R(\alpha^\prime)$. In this case, the flight path of $c_\beta$ on the tower $T_x$ containing $c_\alpha$ does not pass from the cell $c_\alpha$ or the cell just above it. We illustrate this situation with the following picture.

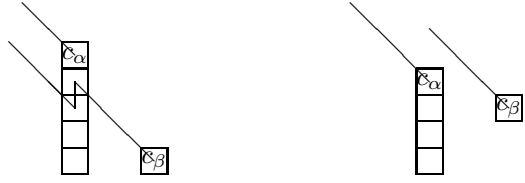
\begin{center}
\begin{picture}(80,60)

\put(20,40){\tableau{{c_\alpha}\\{}\\{}\\{}\\{}}}
\put(50,0){\tableau{{c_\beta}}}
\multiput(0,50)(0,0){1}{\line(1,-1){25}}
\multiput(25,25)(0,0){1}{\line(0,1){10}}
\multiput(25,35)(0,0){1}{\line(1,-1){30}}

\multiput(5,65)(0,0){1}{\line(1,-1){20}}

\end{picture}\hskip.75in
\begin{picture}(80,60)

\put(20,30){\tableau{{c_\alpha}\\{}\\{}\\{}}}
\put(50,20){\tableau{{c_\beta}}}
\multiput(-5,65)(0,0){1}{\line(1,-1){30}}

\multiput(25,55)(0,0){1}{\line(1,-1){30}}

\end{picture}\hskip.15in
\end{center}

In particular, on the tower $T_x$, the distance between the flight paths of
the cells $c_\alpha$ and $c_\beta$ is at least 2. Therefore, in the first case,
by Lemma \ref{lem:ax-zigzag}, we get that $|\alpha_n-\beta_n|\ge 2$. The same
equality is obtained also in the second case by an easy modification of the proof
of Lemma \ref{lem:ax-zigzag}.

Now since $c_\beta$ is a corner cell in shape$(R(\alpha^\prime))$, there is a standard
tower tableau of this shape (with the cell $c_\beta$ labelled by $n-1$) with reading
word $\gamma$ such that $\gamma = \gamma_1\gamma_2\ldots\gamma_{n-2}\beta_n.$
Moreover, by the induction hypothesis, we have $$s_{[\gamma]} = s_{[\alpha^\prime]}.$$

On the other hand, $c_\alpha$ is a corner cell in shape$R(\beta^\prime)$ where
$\beta^\prime = \beta_1\beta_2\ldots\beta_{n-1}$. Hence there is a standard tower
tableau of this shape with reading word $\delta$ such that $\delta =
\delta_1\delta_2\ldots\delta_{n-2}\alpha_n.$
Again, by the induction hypothesis, we have $$s_{[\delta]} = s_{[\beta^\prime]}.$$
Now we have
\[
\mathrm{shape}R(\gamma_1\gamma_2\ldots\gamma_{n-2}) = c_\beta^\nwarrow
c_\alpha^\nwarrow \mathcal T = c_\alpha^\nwarrow c_\beta^\nwarrow \mathcal T =
\mathrm{shape}R(\delta_1\delta_2\ldots\delta_{n-2})
\]
and hence, by the induction hypothesis, we get the equality
$$s_{\gamma_1\gamma_2\ldots\gamma_{n-2}}=s_{\delta_1\delta_2\ldots\delta_{n-2}}.$$
Therefore using the above equalities, we obtain
\begin{align*}
s_{[\alpha]} &= s_{[\alpha^\prime]}s_{\alpha_n} = s_{[\gamma]}s_{\alpha_n}=
s_{[\gamma_1\gamma_2\ldots\gamma_{n-2}]}s_{\beta_n}s_{\alpha_n}\\
&= s_{[\delta_1\delta_2\ldots\delta_{n-2}]}s_{\alpha_n}s_{\beta_n} = s_{[\delta]}s_{\beta_n} = s_{[\beta^\prime]}s_{\beta_n}\\
&= s_{[\beta]}.
\end{align*}
Hence the case is closed.

The final case is the one where $c_\beta$ is not a corner cell in
shape$R(\alpha^\prime)$.
In this case, the flight path of
the cell $c_\beta$ should pass from the cell $c_\alpha$. Indeed, the flight path of $c_\beta$ cannot pass from a
cell below $c_\alpha$ since, then $c_\beta$ is a corner cell in shape$R(\alpha^\prime)$ as seen in the
previous case.
Also note that if the flight path of $c_\beta$ passes from the cell just above $c_\alpha$, then $c_\beta$ cannot
be a corner cell in $\mathcal T$. Any other cell above $c_\alpha$ is also not possible since in this case the
flight path of $c_\beta$ will not be effected by the removal of the cell $c_\alpha$. Thus
the situation, in shape$R(\alpha)$, is as follows.
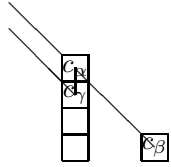
\begin{center}
\begin{picture}(80,60)

\put(20,30){\tableau{{c_\alpha}\\{c_\gamma}\\{}\\{}}}
\put(50,0){\tableau{{c_\beta}}}
\multiput(0,50)(0,0){1}{\line(1,-1){25}}
\multiput(25,25)(0,0){1}{\line(0,1){10}}
\multiput(25,35)(0,0){1}{\line(1,-1){30}}

\multiput(0,60)(0,0){1}{\line(1,-1){20}}

\end{picture}\hskip.75in
\end{center}
Notice that in the above case, the cell $c_\gamma$ has a flight path, since the cell
$c_\beta$ is a corner cell and must contain the flight path of $c_\gamma$ in its
flight path. Thus by Lemma \ref{lem:ax-zigzag}, we get that
\[
|\mathrm{flight}\#(c_\alpha,\mathcal T) - \mathrm{flight}\#(c_\gamma, \mathcal T)| =1.
\]
Also since the flight numbers of $c_\gamma$ and $c_\beta$ are the same, we get
\[
|\mathrm{flight}\#(c_\alpha,\mathcal T) - \mathrm{flight}\#(c_\beta, \mathcal T)| =1.
\]
Therefore, although $c_\beta$ is not a corner cell in $R(\alpha^\prime)$,
the cell $c_\alpha$ is a corner in $R(\beta^\prime)$ and also $c_\gamma$ is a corner
in $c_\alpha^\nwarrow R(\beta^\prime)$. Thus we have the diagram
\[
c_\gamma^\nwarrow c_\alpha^\nwarrow c_\beta^\nwarrow \mathcal T.
\]
Furthermore, if the flight number of $c_\beta$ is $i$, then that of $c_\alpha$ and
$c_\gamma$ should
be $i+1$ and $i$, respectively. Hence if $u$ is a reduced word with recording tableau
of the above shape, we get that
\[
\mathrm{shape}(R(u(i)(i+1)(i))) = \mathrm{shape}(R(\beta)).
\]
Since the last term $\beta_n$ of $\beta$ is $i$, the word $u(i)(i+1)(i)$ is braid related to $\beta$ by the first case.

On the other hand, it is clear that $c_\gamma$ is a corner cell in $R(\alpha^\prime)$.
Moreover when $c_\gamma$ is removed, $c_\beta$ becomes a corner cell of the new
diagram $c_\gamma^\nwarrow c_\alpha^\nwarrow \mathcal T$ and the flight number of
$c_\beta$ in this diagram is the same as the flight number of $c_\alpha$ in
$\mathcal T$. Thus we have the diagram
\[
c_\beta^\nwarrow c_\gamma^\nwarrow c_\alpha^\nwarrow \mathcal T
\]
with the flight numbers of $c_\beta, c_\gamma$ and $c_\alpha$ given respectively by
$i+1, i$ and $i+1$. Moreover we have
\[
c_\gamma^\nwarrow c_\alpha^\nwarrow c_\beta^\nwarrow \mathcal T = c_\beta^\nwarrow
c_\gamma^\nwarrow c_\alpha^\nwarrow \mathcal T
\]
as tower diagrams. Thus we also have
\[
\mathrm{shape}(R(u(i+1)(i)(i+1))) = \mathrm{shape}(R(\alpha))
\]
and hence $u(i+1)(i)(i+1)$ is braid related to $\alpha$ by the first case. Now the
result follows since $u(i+1)(i)(i+1)$ and $u(i)(i+1)(i)$ are also braid related.
\end{proof}

We have thus proved that each tower diagram determines a unique permutation and vice versa a unique tower diagram is determined by a given permutation. The next result
shows that the words that we thus obtain are indeed reduced.
\begin{thm}\label{thm:SRW2RW}
The sliding and recording algorithm does not terminate on the word
$\alpha= \alpha_1\alpha_2\ldots\alpha_n$, that is, $\alpha\in
\mathrm{SRW}(\mathbb{Z}^+)$ if and only if
the word $s_{\alpha_1}s_{\alpha_2}\ldots s_{\alpha_n}$ is a reduced
expression for $s_{[\alpha]}$.
\end{thm}

\begin{proof} We will proceed by induction on the size of the words.
For  $n=1$ both sides of the statement follows directly. So suppose
that both sides of  the statement are true for all words of size
$n-1$ and let $\alpha= \alpha_1\alpha_2\ldots\alpha_n $ be a word on
$\mathbb{Z}^+$.

We first assume that $\alpha \in \mathrm{SRW}(\mathbb{Z}^+)$.
Equivalently  the SR algorithm does not terminate on $\alpha$ and
therefore  it does not terminate the subword
$\alpha_1\alpha_2\ldots\alpha_{n-1}$, and
 by induction argument $s_{\alpha_1}s_{\alpha_2}\ldots
 s_{\alpha_{n-1}}$ is a reduced expression for $s_{[\alpha]} \cdot
 s_{\alpha_n}$. Therefore $l(s_{[\alpha]} \cdot s_{\alpha_n})=n-1$ and $l(s_{[\alpha]})$ is either $n$ or $n-2$.

Now if
 $s_{\alpha_1}s_{\alpha_2}\ldots
 s_{\alpha_{n}}$ is not a reduced expression then $l(s_{[\alpha]})<n$ and by previous argument it must be equal to $n-2$. So
we have $l(s_{[\alpha]})<l(s_{[\alpha]} \cdot s_{\alpha_n})$ and hence
 $s_{[\alpha]}\cdot s_{\alpha_n}$  has another
 reduced expression say $s_{\alpha'_1}s_{\alpha'_2}\ldots
 s_{\alpha'_{n-1} }$  Coxeter related to $s_{\alpha_1}s_{\alpha_2}\ldots
 s_{\alpha_{n-1}}$ where $s_{\alpha'_{n-1}}=s_{\alpha_{n}}$.
 Now by
the previous theorem, we have
 $$S(\alpha_1 \alpha_2 \ldots
 \alpha_{n-2} \alpha_{n-1})=S=S(\alpha'_1 \alpha'_2 \ldots \alpha'_{n-2}
 \alpha_{n} )
 $$

Let $(i,j)$ be the cell in $S$ which is obtained by sliding $\alpha_n$ in
to the diagram $S(\alpha'_1 \alpha'_2 \ldots \alpha'_{n-2})$. Then $(i,j)$ must
be a corner cell in $S$. On the other hand $ S(\alpha_1 \alpha_2 \ldots
 \alpha_{n-2} \alpha_{n-1}{\alpha_n})$ is obtained by sliding another  $\alpha_n$ in to $S$ but then this sliding
 must go through the corner cell $(i,j)$ of $S$. But since the top of $(i,j)$ is empty in  $S$ this shows that the SR algorithm
 terminates for $\alpha_1 \alpha_2 \ldots \alpha_{n-1}\alpha_n$ which is a contradiction.

We now assume that $\alpha= \alpha_1\alpha_2\ldots\alpha_n  \not \in
\mathrm{SRW}(\mathbb{Z}^+)$ and let $k>1$ be the integer such that
SR algorithm does not terminate on $ \alpha_1\alpha_2\ldots\alpha_k
$  but it terminates on    $
\alpha_1\alpha_2\ldots\alpha_k\alpha_{k+1}   $.  Therefore in SR
algorithm    $\alpha_{k+1}$ is slided through some corner cell $(i,j)$
of $S(\alpha_1\alpha_2\ldots\alpha_k)$. Therefore one can obtain
another word $\alpha'_1\alpha'_2\ldots\alpha'_k$, on which SR
algorithm produces the same tableau
$S(\alpha_1\alpha_2\ldots\alpha_k)$ by producing  the  cell
$(i,j)$ at the end and  therefore $\alpha'_k=\alpha_{k+1}$.  On the
other hand  by  the induction argument  $s_{\alpha_1} s_{\alpha_2}\ldots
s_{\alpha_k}$ and $s_{\alpha'_1} s_{\alpha'_2}\ldots s_{\alpha'_k}$
are Coxeter related and so are $s_{\alpha_1} s_{\alpha_2}\ldots
s_{\alpha_k}s_{\alpha_{k+1}}\ldots s_{\alpha_n}$ and $s_{\alpha'_1}
s_{\alpha'_2}\ldots s_{\alpha'_k}s_{\alpha_ {k+1}}\ldots
s_{\alpha_n}$. Now the length of $s_{\alpha'_1} s_{\alpha'_2}\ldots
s_{\alpha'_k} s_{\alpha_{k+1}}\ldots s_{\alpha_n}$ is clearly less
than or equal to $n-2$ since $ s_{\alpha'_k} =s_{\alpha_{k+1}}$.
Hence  the length of
$\alpha_1\alpha_2\ldots\alpha_k\alpha_{k+1}\ldots\alpha_n$ is less
than or equal to $n-2$ and it  is not reduced.
\end{proof}

Throughout the paper, we have introduced several functions. Together with the known
functions and relations, we obtain the diagram in the introduction. We repeat the
diagram by using our notations.

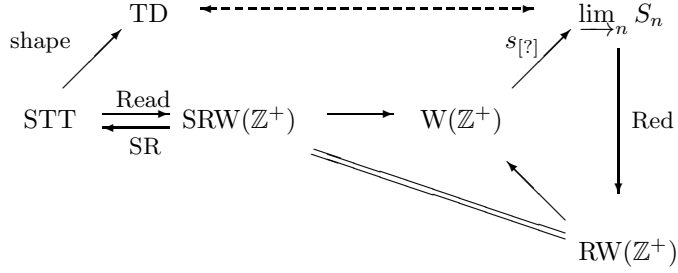
\begin{center}
\begin{picture}(160,120)

\put(-30,90){$\mathrm{TD}$}
\put(140,90){$\varinjlim_n S_n$}
\put(-70,50){$\mathrm{STT}$}
\put(-10,50){$\mathrm{SRW}(\mathbb Z^+)$}
\put(80,50){$\mathrm{W}(\mathbb Z^+)$}
\put(140,0){$\mathrm{RW}(\mathbb Z^+)$}

\put(112,80){{$s_{[?]}$}}
\multiput(115,65)(0,0){1}{\vector(1,1){22}}

\multiput(0,95)(5,0){25}{\line(1,0){3}}
\multiput(0,95)(5,0){1}{\vector(-1,0){3}}
\multiput(120,95)(5,0){1}{\vector(1,0){3}}

\put(160,50){\small{Red}}
\multiput(155,80)(0,0){1}{\vector(0,-1){55}}

\multiput(135,15)(0,0){1}{\vector(-1,1){22}}

\multiput(135,10)(0,0){1}{\line(-3,1){95}}
\multiput(135,8)(0,0){1}{\line(-3,1){96}}


\put(-75,80){\small{shape}}
\multiput(-55,65)(0,0){1}{\vector(1,1){22}}

\put(-35,58){\small{Read}}
\multiput(-40,55)(0,0){1}{\vector(1,0){25}}
\put(-30,40){\small{SR}}
\multiput(-15,50)(0,0){1}{\vector(-1,0){25}}
\multiput(45,55)(0,0){1}{\vector(1,0){25}}
\end{picture}
\end{center}

Although the diagram is self-expository, there are
several other results that we can draw from the diagram.

Let $\omega\in S_n$ be a permutation and let $\alpha$ be a reduced word representing
$\omega$. Also let $\mathcal T_\omega$ be the shape of the standard tower tableau
$R(\alpha)$. We call $\mathcal T_\omega$ the \textbf{\textit{shape}} of $\omega$.

We denote by Red$(\omega)$ the set of all reduced words representing $\omega$ and
by STT$(\omega)$ the set of all standard tower tableau of shape $\omega$.

With these notations, by Theorem \ref{thm:srw-stt}, the reading function Read$:\mathrm{STT}\rightarrow W(\mathbb Z^+)$ restricts to a bijection
\[
\mathrm{Read}: \mathrm{STT}\rightarrow \mathrm{RW}.
\]
Moreover, by Theorem \ref{thm:braid2shape} and Theorem \ref{thm:shape2braid}, the
bijection further specialize to a bijection
\[
\mathrm{Read}: \mathrm{STT}(\omega)\rightarrow \mathrm{Red}(\omega)
\]
as promised at the beginning of this section. Note that, in both cases the inverse
is given by the inverse of the reading function, that is, by the SR algorithm.

We state this result as a theorem, for future reference.
\begin{thm}\label{thm:red2stt}
Let $\omega$ be a permutation. Then the reading function and the SR algorithm are
inverse bijective correspondences between
\begin{enumerate}
\item the set $\mathrm{Red}(\omega)$ of reduced expressions for $\omega$ and
\item the set $\mathrm{STT}(\omega)$ of standard tower tableaux of shape
$\omega$.
\end{enumerate}
\end{thm}
As a corollary, we get the dotted connection shown in the above diagram.
\begin{cor}\label{cor:perm2td}
There is a bijective correspondence between
\begin{enumerate}
\item the set $\varinjlim_n S_n$ of all finite permutations and
\item the set $\mathrm{TD}$ of all finite tower diagrams
\end{enumerate}
given by $\omega\rightarrow \mathrm{shape}(\omega)$.
\end{cor}

To use the above bijection as a tool to describe reduced words for a given permutation
$\omega$, one needs to know all standard tower tableaux of the determined shape.
This can be done recursively, as follows.

Let $\mathcal T= \mathcal T_\omega$ be a tower diagram. We denote the set of all
corner cells of $\mathcal T$ by $C(\mathcal T)$. By the definition of a standard tower
tableaux and the reading function, any cell in $C(\mathcal T)$ corresponds to a
terminal term in some reduced word in Red$(\omega)$. Therefore we have
$$
\mathrm{STT}(\mathcal T) = \bigsqcup_{c\in C(\mathcal T), R\in \mathrm{STT}
(\mathcal T - c)} \{ R\cup \{[c,n]\} \}.
$$
It is clear that the above equality produces all standard tower
tableaux of the shape $\mathcal T$ and hence, applying the reading function, all reduced words for the permutation $\omega$.

Notice that the above remark suggests to determine all standard tower tableaux by
removing one corner cell at a time. As it is, although the algorithm is very
systematic, it is also slow. However the advantage of our algorithm is that we
determine the shape of the diagram at the beginning. In a sequel to this paper, we
will introduce a faster algorithm to determine the set of all reduced expressions
of a given permutation.
\section{Natural Word of a Permutation}
\label{Section:NaturalWord}
In this section, we introduce a canonical reduced word for any permutation.
This word has several nice properties and will be used in Section \ref{section:tower2rothe} in a crucial way.

We are still denoting by $\mathcal T$ an arbitrary tower diagram. The \textbf{\textit{ natural labelling}} $\mathbb T$ of $\mathcal T$ is the labelling defined inductively
as follows. If $\mathcal T$ contains a unique cell, then the unique labelling of
$\mathcal T$ is natural. Otherwise, if $\mathcal T$ contains $n$ cells, we remove
the cell on the top of the left most tower of $\mathcal T$ and determine the natural
labelling of the new diagram. Then we add the removed cell into its original position
with the label $n$.

In other words, we label $\mathcal T$ by $1,2,\ldots, n$ in the increasing manner
starting from the bottom cell of the right most tower and then by going first from
bottom to top and then right to left. It is clear
that the natural labelling of a tower diagram is standard and hence the word
Read$(\mathbb T)$ is defined. We call Read$(\mathbb T)$ the \textbf{\textit{natural
word}} for the diagram $\mathcal T$. If $\mathcal T= \mathcal T_\omega$, then we write
$$\eta_\mathcal T:=\eta_\omega:= \mathrm{Read}(\mathbb T)$$ and say that $\eta_\omega$ is the natural
word for $\omega$.

The natural labelling of the tower diagram $\mathcal T = (2,1,0,1)$ is the tableau
$$\line(0,1){30}\line(1,0){50}
\multiput(-50,0)(10,0){5}{\line(0,1){2}}
\multiput(-50,0)(0,10){3}{\line(1,0){2}}
\put(-50,20){\tableau{\\{4}\\{3}}} \put(-40,10){\tableau{\\{2}}}
\put(-20,10){\tableau{\\{1}}}
$$
and the corresponding natural word is $\eta_{(2,1,0,1)} = 4\, 2\, 1\, 2.$

As an other example, consider the longest word $\omega_0$ in $S_4$. The commonly used
reduced expression for $\omega_0$ is
$3\, 2\, 1\, 3\, 2\, 3$ and it is easy to see that the associated tower diagram and
its natural labelling is given by
$$\line(0,1){50}\line(1,0){50}
\multiput(-50,0)(10,0){5}{\line(0,1){2}}
\multiput(-50,0)(0,10){3}{\line(1,0){2}}
\put(-50,30){\tableau{\\{6}\\{5}\\{4}}} \put(-40,20){\tableau{\\{3}\\{2}}}
\put(-30,10){\tableau{\\{1}}}
$$
Therefore the natural word for the longest permutation is $3\, 2\, 3\, 1\, 2\, 3.$
The natural word $\eta_\omega$ can be characterized by certain properties given below.
Write the tower $\mathcal T$ as a concatenation
$$\mathcal T = \bigsqcup_{i=1}^k(\mathcal T_i)$$  of its towers.
Then the natural word $\eta_\omega$ decomposes as
$$\eta_\omega = \bigsqcup_{i=1}^{k}\eta_i$$ where $\eta_i$ is the natural word of the
tower $\mathcal T_i$ and we agree that when $\mathcal T_i$ is empty, the corresponding
natural word is also empty.

Now if $|\mathcal T_i| = k_i>0$, then we have
\[
\eta_i = i\, (i+1)\, \ldots\, (i+k_i).
\]
In particular, for each $i$, the word $\eta_i$ is an increasing sequence of
consecutive integers. On the other hand, if we write $\eta_i =
\eta_i^1\eta_i^2\ldots\eta_i^{k_i}$ then the sequence $\eta_1^1\eta_2^1\ldots\eta_s^1$
is strictly decreasing.

It is also clear that the properties characterize the natural word in the following
sense. When a tower diagram is given, then the above constructed word is the unique
reading word associated to the diagram with the specified properties. Therefore
we have proved the following result.
\begin{pro}
Let $\omega\in S_n$ be a permutation. Then there is a unique reduced expression $\eta$
representing $\omega$ such that the word $\eta$ is a sequence
$\eta_1\eta_2\ldots\eta_s$ of increasing subsequences where
\begin{enumerate}
\item each subsequence $\eta_i$ is a sequence of consecutive integers and
\item the sequence of initial terms of $\eta_i$'s is strictly decreasing.
\end{enumerate}
\end{pro}
\begin{rem}
There is another canonical choice for a standard labelling of a tower diagram where
we label the cells first from right to left and then bottom to top. For example, such
labelling of the tower diagram of the longest word in $S_4$ is given as follows.
$$\line(0,1){50}\line(1,0){50}
\multiput(-50,0)(10,0){5}{\line(0,1){2}}
\multiput(-50,0)(0,10){3}{\line(1,0){2}}
\put(-50,30){\tableau{\\{6}\\{5}\\{3}}} \put(-40,20){\tableau{\\{4}\\{2}}}
\put(-30,10){\tableau{\\{1}}}
$$
Notice that the above labelling is also standard, by definition. In the above example,
the reading word of the tower tableau is
$3\, 2\, 1\, 3\, 2\, 3.$
Note that this expression is the same as the commonly used one.
\end{rem}
\section{From Rothe diagrams to tower diagrams}
\label{section:rothe2tower}
In this section, we show how the tower diagram can be obtained
from the corresponding Rothe diagram. The idea is that when the Rothe diagram of
a permutation is given, the corresponding tower diagram can be
determined by pushing all the nodes of the Rothe diagram to the top
row and then reflecting the resulting diagram on the top border of
the Rothe diagram. In other words, to obtain the tower diagram, we
forget the gaps between the cells in the columns of the Rothe
diagram. However, we should note that, by Corollary \ref{cor:perm2td}, forgetting
gaps is not very crucial in determining the permutation corresponding to the
Rothe diagram. In the next section, we further prove
that the tower diagram determines the Rothe diagram.

To state the main result of this section, we first recall from \cite{M} some
basic facts concerning Rothe diagrams. We begin by the definition.
Let $\omega$ be a permutation in $S_n$. The \textbf{\emph{Rothe diagram}}
$D_\omega$ of $\omega$ is the set
\[
D_\omega = \{ (i,j)| 1\le i,j\le n, \omega(i)< j, \omega^{-1}(j)<i \}.
\]
We sketch the set $D_\omega$ as an $n\times n$-array with the
points in $D_\omega$ marked with $\circ$. Another way to
determine the Rothe diagram $D_\omega$ can be described as follows. Let $D$ be an
 $n\times n$-array with empty cells.
Mark the cell $(i,j)$ with a cross if $\omega(i) = j$ and leave
it empty otherwise. Then for each crossed cell $(i,j)$, mark the cells
of the hook with vertex $(i,j)$ by a dot. Now the pairs in $D_\omega$ are
those which are left empty at the end of this process.

\begin{ex}
Let $\omega = 214635$. Then the array $D$ with marks is the following
and the Rothe diagram $D_\omega$ of $\omega$ are given by the following diagrams.

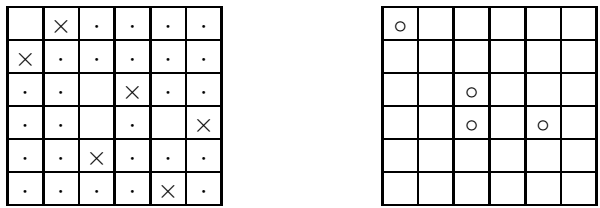
\begin{align*}
\begin{Young}
&$\times$ & $\cdot$ &$\cdot$ &$\cdot$ &$\cdot$\cr
$\times$ & $\cdot$ &$\cdot$ &$\cdot$ &$\cdot$ &$\cdot$\cr
$\cdot$ &$\cdot$& & $\times$&$\cdot$ &$\cdot$\cr
$\cdot$ &$\cdot$& &$\cdot$& &$\times$\cr
$\cdot$ &$\cdot$&$\times$ &$\cdot$&$\cdot$ &$\cdot$\cr
$\cdot$ &$\cdot$&$\cdot$ &$\cdot$&$\times$ &$\cdot$\cr
\end{Young}&&&
\begin{Young}
$\circ$&&&&&\cr
&&&&&\cr
&&$\circ$&&&\cr
&&$\circ$&&$\circ$&\cr
&&&&&\cr
&&&&&\cr
\end{Young}
\end{align*}
\end{ex}

Now let $s_i = (i,i+1)$ be a standard transposition. Next we
determine the effect of multiplication of $\omega$ by $s_i$ on the
Rothe diagram. Part of this result is stated in \cite[Lemma 4.6]{FGRS}. We provide
a complete proof.

Let $\omega$ and $s_i$ be as above and let $\tilde\omega= \omega
s_i$. Let $D_\omega(\times)$ be the $n\times n$- array marked with
$\times$ as described in the previous section. Then $D_{\tilde\omega}(\times)$ is obtained by
applying $s_i$ to the rows of $D_\omega(\times)$, by definition.
Recall that the length $l=l(\omega)$ of $\omega$ is equal to the
size $|D_\omega|$ of its Rothe diagram and also that
$l(\tilde\omega) = l(\omega) \pm 1$. Thus $|D_{\tilde\omega}| = |D_\omega| \pm 1$.

It is possible to determine $D_{\tilde\omega}$ more explicitly. There are two cases to consider. First assume that
the $\times$ in the $i$-th row is on the right of the one on the $(i+1)$-st as
illustrated in the following picture.

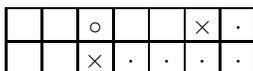
\begin{align*}
\begin{Young}
$$ &&$\circ$ && &$\times$&$\cdot$\cr
$$ &&$\times$ &$\cdot$&$\cdot$ &$\cdot$&$\cdot$\cr
\end{Young}
\end{align*}

\noindent Here we leave a cell empty if we are not sure about its
content with the given information. It is easy to observe that the circled cell in the
above diagram is contained in the Rothe diagram. Now multiplication by
$s_i$ gives the following partial diagram.

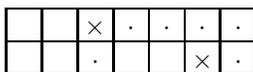
\begin{align*}
\begin{Young}
&&$\times$ &$\cdot$&$\cdot$ &$\cdot$&$\cdot$\cr
&&$\cdot$ && &$\times$&$\cdot$\cr
\end{Young}
\end{align*}

\noindent Notice that this operation has no effect outside the rectangle determined
by the crosses on the $i$-th and $(i+1)$-st row.
Therefore we can concentrate on the following partial diagrams.

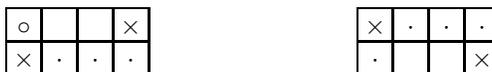
\begin{align*}
\begin{Young}
$\circ$ && &$\times$\cr
$\times$ &$\cdot$&$\cdot$ &$\cdot$\cr
\end{Young}&&&
\begin{Young}
$\times$ &$\cdot$&$\cdot$ &$\cdot$\cr
$\cdot$ & & &$\times$\cr
\end{Young}
\end{align*}

Now the circled cell in the first partial diagram is contained in
the Rothe diagram $D_\omega$ and the empty cells of the first
diagram may or may not be contained in $D_\omega$. On the other
hand, the second partial diagram shows that after applying $s_i$,
the empty cells are moved to a lower row and their contents are not
changed, but the circled cell is moved to a lower row and the
content is changed to a dot. Therefore the only change in
$D_\omega$ after applying $s_i$ is to remove the circled cell
shown above and to move the empty cells one row below. Therefore
in this case, the length $l(\tilde\omega)$ of $\tilde\omega$ is
decreased.

The other case where the $\times$ in the $i$-th row is on the left
of the one on the $(i+1)$st is similar to the above case and
corresponds to the case where the length of $\tilde\omega$ is
increased. Therefore we have proved the proposition below.
\begin{pro}\label{pro:actionRothe}
Let $\omega$ be a permutation in $S_n$ and let $D$ be its Rothe
diagram. Let $s_i\in S_n$ be an adjacent transposition. Let
$D^\prime$ be the Rothe diagram of $\omega\, s_i$. Then
\begin{enumerate}
\item the equality $l(\omega\, s_i) = l(\omega) -1$ holds if and only if
$D^\prime$ is obtained from $D$ by first removing the element $(i,\omega(i+1))$ of $D$
and then interchanging the $i$-th and
$(i+1)$-st rows of $D$, and
\item the equality $l(\omega\, s_i) = l(\omega) + 1$ holds if and only if
$D^\prime$ is obtained from $D$ by first interchanging
the $i$-th and $(i+1)$-st rows of $D$ and then adding the pair
$(i,\omega(i))$ to (the $i$-th row of) the new diagram.
\end{enumerate}
\end{pro}
Now we are ready to show how the Rothe diagram determines the tower diagram.
Let $\omega$ be a permutation and $\alpha$ be a reduced expression representing $\omega$. Let $\mathcal T_\omega$
(resp. $D_\omega$) be the tower (resp. Rothe) diagram associated to $\omega$. When there is no risk of confusion, we
omit the subscripts. Write
\[
\mathcal T = \bigsqcup_{i=1}^n (\mathcal T_i) \;\;\; \mbox{\rm and}\;\;\; D = \bigsqcup_{j= 1}^m D^j
\]
where the first sum (resp. second sum) is over the columns of $\mathcal T$ (resp. $D$)
from left to right and where $\mathcal T_i$ (resp. $D^i$) is the $i$th column of
$\mathcal T$ (resp. $D$). Finally we have the following correspondence, as promised at
the beginning of this section.
\begin{thm}\label{thm:rothe2tower}
Let $\omega,\mathcal T$ and $D$ be as above. Then $n = m$ and for each $1\le i\le n$,
we have $|\mathcal T_i| = |D_i|$.
\end{thm}
\begin{proof}
We argue by induction on the length $l= l(\omega)$ of $\omega$.
The case $l=1$ is trivial. We assume the result for $l-1$ and let
$\omega$ be of length $l$. Let $(i_0, j_0)$ be the lowest cell in the first non-empty
column of $D_\omega$ as illustrated with a bullet in the following diagram.
\begin{align*}
\begin{Young}
$\times$&$\cdot$ &$\cdot$ &$\cdot$ &$\cdot$ &$\cdot$ &$\cdot$ \cr
$\cdot$&$\times$ &$\cdot$ &$\cdot$ &$\cdot$ &$\cdot$ &$\cdot$ \cr
$\cdot$&$\cdot$ &$\circ$ &? &? &? &? \cr
$\cdot$&$\cdot$ &$\bullet$&? &? &? &? \cr
$\cdot$&$\cdot$ &$\times$ &$\cdot$ &$\cdot$ &$\cdot$ &$\cdot$ \cr
$\cdot$&$\cdot$ &$\cdot$ &? &? &? &? \cr
\end{Young}
\end{align*}
Now let $s= s_{i_0}$. Then we claim that $l(\omega s)= l-1$.
Indeed, by the choice of $i_0$, there is a cross just below
$(i_0,j_0)$ and the cross on the $i_0$-th row is on the right of this
one. Thus by Proposition \ref{pro:actionRothe}, the Rothe diagram of $D_{\omega s}$ is
obtained by removing the cell
$(i_0,j_0)$ and possibly, moving some other cells one row below and hence the length
decreases. Thus by the induction hypothesis, the tower diagram $\mathcal T_{\omega s}$
and the Rothe diagram $D_{\omega s}$ satisfy the conclusion of
the theorem. Thus
\[
\mathcal T_{\omega s} = \bigsqcup_{i=1}^n (\mathcal T_i^\prime) \;\;\; \mbox{\rm and}
\;\;\; D_{\omega s} = \bigsqcup_{i= 1}^n D_j^\prime
\]
where for each $i$, we have $|\mathcal T_i^\prime|=|D_i^\prime|$. Now if we apply $s$
once more to the Rothe diagram
$D_{\omega s}$ we will obtain $D$. Furthermore $j_0$ is the unique integer such that
$|D_{j_0}^\prime| \neq |D_{j_0}|$ and we have $|D_{j_0}^\prime| = |D_{j_0}| + 1$.
On the other hand, the first non-empty column of $T_{\omega s}$ is $\mathcal T_{j_0}$,
since the Rothe diagram of $\omega s$ has this property. Thus the number of boxes on
this column is
$i_0 - j_0$. Therefore if we slide $i_0$ to the tower diagram $\mathcal T_{\omega s}$,
it will stop on the top of the tower $\mathcal T_{j_0}$, as required.
\end{proof}
\section{From tower diagrams to Rothe diagrams}
\label{section:tower2rothe}
Given a permutation $\omega$, we write $\mathcal T_\omega$ for the tower
diagram of $\omega$ and $n_\omega$ for the natural word associated
to $\mathcal T_\omega$. In this section, we prove that it is possible to
recover the Rothe diagram $D_\omega$ from $\mathcal T_\omega$. Again, when
there is no risk of confusion, we omit the subscripts. It follows
from the previous section that the tower diagram $\mathcal T$ determines
the number of boxes on any given column of the Rothe diagram. To
recover the Rothe diagram, we need to determine the vertical
positions of the boxes of the non-empty columns. In order to
achieve this, we introduce \emph{virtual sliding} of words, the
\emph{virtual tower diagram} and the \emph{complete tower
diagram} of a permutation.

Let $\alpha$ be a reduced word for the permutation $\omega$. Write
$$\alpha = \alpha_1\, \alpha_2\, \ldots\, \alpha_l$$ as a product of
transpositions where $l$ is the length of $\omega$. The
\emph{virtual sliding algorithm} is obtained by extending the
sliding algorithm to negative integers as follows. Recall that the
sliding of $\alpha$ slides the word $\alpha_1\, \alpha_2\, \ldots\, \alpha_l$
along the lines $y = -x +\alpha_j$ as $j$ runs from $1$ to $l$ where
the line $y = 0$ is a border for this sliding.

On the other hand, the \textbf{\emph{virtual sliding}} of the word $\alpha$ slides the
$l$-tuple  $(-\alpha_l,-\alpha_{l-1}, \ldots,-\alpha_1)$
along the lines
$y = -x -\alpha_j$ as $j$ runs from $l$ to $1$ and the border for this sliding is the line
$x =0$. We agree that the rules of
sliding explained in Section \ref{Section:SlidingIntoTower} also apply to the virtual
sliding. The diagram $\mathcal T^-$ obtained at the
end is called the \textbf{\emph{virtual tower diagram}} of $\alpha$.

\begin{ex} Let $\alpha = 3452312$. Then the virtual tower
diagram of $\mathcal T$ of $\alpha$ is given as follows.
$$
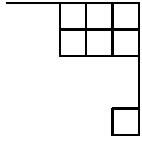
\line(0,-1){50}
\line(-1,0){50}
\put(-30,0){\tableau{\\{}}}
\put(-20,0){\tableau{\\{}}}
\put(-10,0){\tableau{\\{}}}
\put(-30,-10){\tableau{\\{}}}
\put(-20,-10){\tableau{\\{}}}
\put(-10,-10){\tableau{\\{}}}
\put(-10,-40){\tableau{\\{}}}
$$
Now we put the tower and the virtual tower diagrams of $\alpha$ together as follows.
$$
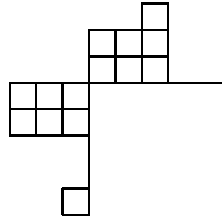
\put(0,0){\line(0,-1){50}}
\put(0,0){\line(-1,0){30}}
\put(0,0){\line(0,1){10}}
\put(0,0){\line(1,0){50}}
\put(0,20){\tableau{\\{}\\{}}}
\put(10,20){\tableau{\\{}\\{}}}
\put(20,30){\tableau{\\{}\\{}\\{}}}
\put(-30,0){\tableau{\\{}}}
\put(-20,0){\tableau{\\{}}}
\put(-10,0){\tableau{\\{}}}
\put(-30,-10){\tableau{\\{}}}
\put(-20,-10){\tableau{\\{}}}
\put(-10,-10){\tableau{\\{}}}
\put(-10,-40){\tableau{\\{}}}
$$
\end{ex}
The above diagram is now called the \textbf{\emph{complete tower diagram}} of $\alpha$
and we denote it by $\Upsilon_{\omega} = (\mathcal T_\alpha,\mathcal T_\alpha^-)$.

\begin{rem} It is easy to see that the virtual tower diagram for $\alpha$ is obtained from the tower diagram of the
reverse word $\alpha^{-1}$ of $\alpha$ by reflecting it along the
line $y = -x$. Hence the results of the previous sections are also
valid for virtual tower diagrams. In particular, the complete
tower diagram is uniquely determined by the permutation $\omega$
and conversely, the complete tower diagram determines a unique
permutation. However, since the virtual tower diagram is
determined by the tower diagram, arbitrary tower diagrams cannot
be joined to give a complete diagram. This reflects the similar
property of the Rothe diagrams. This is the reason for not
introducing the virtual sliding at the beginning: we would like to
have the freedom to chose the tower diagram.
\end{rem}
\begin{rem} Similar to Theorem \ref{thm:rothe2tower}, we can also prove that the
virtual tower diagram of $\omega$ is
obtained by pushing all the nodes of the Rothe diagram to
left and then reflecting everything on the left border of
the Rothe diagram. We leave the proof as an exercise to the
interested reader.
\end{rem}

Let $\Upsilon= (\mathcal T, \mathcal T^-)$ be a complete tower diagram. A \textbf{\emph{standard labelling}} of $\Upsilon$ is a pair $(f,f^-)$ of functions
where $f$ is a standard labelling of the tower and $f^-$ is a standard labelling of the virtual tower such that the
reading word of $f$ is reverse to the reading word of $f^-$. In the particular case where $f$ is the natural label for
the upper tower, we call the pair $(f,f^-)$ the \textbf{\emph{natural labelling}} of the complete tower diagram. Moreover a
complete tower diagram with a standard label is called a \textbf{\textit{standard complete tower tableau}}.

In the above example, the natural label for the word $\alpha = 3452312$ is given as
follows.
$$
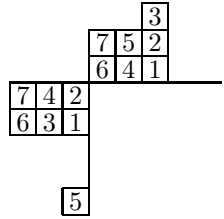
\put(0,0){\line(0,-1){50}}
\put(0,0){\line(-1,0){30}}
\put(0,0){\line(0,1){10}}
\put(0,0){\line(1,0){50}}
\put(0,20){\tableau{\\{7}\\{6}}}
\put(10,20){\tableau{\\{5}\\{4}}}
\put(20,30){\tableau{\\{3}\\{2}\\{1}}}
\put(-30,0){\tableau{\\{7}}}
\put(-20,0){\tableau{\\{4}}}
\put(-10,0){\tableau{\\{2}}}
\put(-30,-10){\tableau{\\{6}}}
\put(-20,-10){\tableau{\\{3}}}
\put(-10,-10){\tableau{\\{1}}}
\put(-10,-40){\tableau{\\{5}}}
$$
Observe, in the above virtual tower diagram, that the row just below the cell labelled
by 1 is empty. The next lemma states that this is true in general when the complete
diagram has natural labels. It is also easy to produce examples of other standard
labelling for which the result does not hold. The proof is deferred to \ref{Appendix}.
\begin{lem}\label{lem:noZigzag}
Let $(T,T^-)$ be a natural complete tower tableau. Let $C$ be the cell in $T^-$ with
label $1$. Then the row just below the cell $C$ is empty.
\end{lem}

As in the case of Rothe diagrams, we want to determine the effect
of the multiplication of $\omega$ by an adjacent transposition
$s_i$ from the right on the complete tower diagram with natural
labels. Actually we only need to determine the effect of
multiplication by the last term of the natural word $\eta$ of
$\omega$. We have the following result.
\begin{pro}\label{pro: actionOfTransII}
Let $\eta = \eta_1\eta_2\ldots\eta_l$ be the natural word for the permutation
$\omega$ and let $s =s_{\eta_l}$. Then the complete tower diagram of
$\omega s$ is obtained from that of $\omega$ in two steps as follows.
\begin{enumerate}
\item From the natural tower tableau of $\omega$, we remove the cell with label
$l$.
\item From the natural virtual tower tableau of $\omega$, we remove the cell
labelled by $1$ and slide the cells on the left of this cell,
if any exists, further to the next row.
\end{enumerate}
\end{pro}
\begin{proof}
The first step is trivial. Being the last term of the word, the removed cell is a
corner. Therefore the removal of the cell with label $l$ corresponds to multiplication
by $s_{\eta_l}$.

The second step, the effect on the virtual tower is less obvious. Let $c$ be the
cell in the virtual tower tableau with label $1$ and assume that it is on the $j$-th row.
Thus removing $\eta_l$ from $\eta$ corresponds to the removal of the cell $c$ from the virtual tableau.
Note that, by Lemma \ref{lem:noZigzag}, the $(j+1)$-st row is empty.

Now if there is no cell on the $j$-th row, then removal of the cell $c$ will not effect the rest of the diagram and we
are done. Thus we assume that $c$ is not the unique cell in its row. It is clear that the only effect of the removal of
$c$ is on the remaining cells in the $j$-th row and the effect is to slide these cells to the next row, as required.
\end{proof}

Finally we are ready to prove the main result of this section. First we
explain the algorithm. Given a complete tower diagram $\Upsilon
=(T,T^-)$ with natural labels. Let $l$ be the size of $T$. Construct the set
$$I = \{ (u,v) : ([u,i],[(v,j])\in T\times T^-, i+j = l+1 \}$$
of pairs of cells from the complete tower diagram whose labels sum up to $l+1$. Then
for each
$(u,v)\in I$, construct the vertical shadow of the cell $u$ and
the horizontal shadow of the cell $v$. Observe that these shadows intersect at the
point $(u_1,-v_2)$. We illustrate this below.
$$
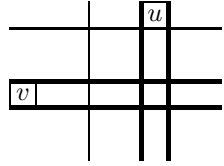
\put(0,0){\line(0,-1){50}}
\put(0,0){\line(-1,0){30}}
\put(0,0){\line(0,1){10}}
\put(0,0){\line(1,0){50}}
\put(-30,-20){\tableau{\\{v}}}
\put(20,10){\tableau{\\{u}}}
\put(20,10){\linethickness{0.5mm}\line(0,-1){60}}
\put(30,10){\linethickness{0.5mm}\line(0,-1){60}}
\put(-30,-20){\linethickness{0.5mm}\line(1,0){80}}
\put(-30,-30){\linethickness{0.5mm}\line(1,0){80}}
$$
Finally construct the set $$R_\omega = \{ (u_1,-v_2)\, :\, (u,v)\in I, u=(u_1,u_2), v=(v_1,v_2)\}.$$
We call $R_\omega$ the \textbf{\textit{Rothification}} of the complete tower diagram
of $\omega$.
\begin{ex} The Rothification of the complete tower diagram in the previous example is
the following diagram. Note that
we put bullets to distinguish the cells in the Rothification.
$$
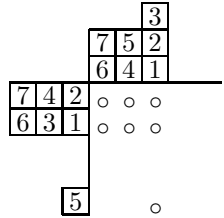
\put(0,0){\line(0,-1){50}}
\put(0,0){\line(-1,0){30}}
\put(0,0){\line(0,1){10}}
\put(0,0){\line(1,0){50}}
\put(0,20){\tableau{\\{7}\\{6}}}
\put(10,20){\tableau{\\{5}\\{4}}}
\put(20,30){\tableau{\\{3}\\{2}\\{1}}}
\put(-30,0){\tableau{\\{7}}}
\put(-20,0){\tableau{\\{4}}}
\put(-10,0){\tableau{\\{2}}}
\put(-30,-10){\tableau{\\{6}}}
\put(-20,-10){\tableau{\\{3}}}
\put(-10,-10){\tableau{\\{1}}}
\put(-10,-40){\tableau{\\{5}}}
\put(0,-10){$\;\circ$}
\put(10,-10){$\;\circ$}
\put(20,-10){$\;\circ$}
\put(0,-20){$\;\circ$}
\put(10,-20){$\;\circ$}
\put(20,-20){$\;\circ$}
\put(20,-50){$\;\circ$}
$$
Now the permutation corresponding to the word $\alpha = 3452312$ is $\omega = 451263$ and its
Rothe diagram is given by the following diagram. Again the cells of the Rothe diagram is marked with a bullet.
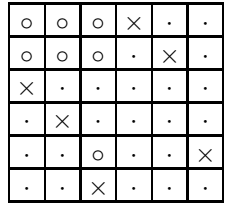
\begin{align*}
\begin{Young}
$\circ$&$\circ$ &$\circ$ &$\times$ &$\cdot$ &$\cdot$  \cr
$\circ$&$\circ$ &$\circ$ &$\cdot$ &$\times$ &$\cdot$ \cr
$\times$&$\cdot$ &$\cdot$ &$\cdot$ &$\cdot$ &$\cdot$\cr
$\cdot$&$\times$ &$\cdot$&$\cdot$ &$\cdot$  &$\cdot$ \cr
$\cdot$&$\cdot$ &$\circ$ &$\cdot$ &$\cdot$ &$\times$ \cr
$\cdot$&$\cdot$ &$\times$ &$\cdot$ &$\cdot$ &$\cdot$ \cr
\end{Young}
\end{align*}
\end{ex}
The above example indicates that the Rothification of the complete tower diagram of $\omega$ might give us the Rothe
diagram of $\omega$. Next we prove this.
\begin{thm}\label{thm:rothefication}
Let $w$ be a permutation and let $\Upsilon_\omega$ be the complete tower diagram of $\omega$
with natural labels.
Then $R_\omega = D_\omega$.
\end{thm}
\begin{proof}
We argue by induction on the length $l$ of $\omega$. The case
$l=1$ is trivial. Let $\eta=\eta_1\eta_2\ldots\eta_l$ be the natural word of $\omega$ Let $\tilde\omega =
\omega\, s_{\eta_l}$. Then it is clear that the natural word of $\tilde\omega$ is $\eta_1\eta_2\ldots\eta_{l-1}$. By
induction, the Rothification of the complete tower diagram of $\tilde\omega$ is equal to the Rothe diagram of
$\tilde\omega$. We obtain these diagrams as follows.

To obtain the Rothefication of the complete tower diagram of $\tilde\omega$, we use Proposition \ref{pro:
actionOfTransII}. Assume that the cell with label $l$ in the tower diagram $\mathcal T$ of $\omega$ is on the $i$-th
column and the cell with label $1$ in the virtual tower diagram $\mathcal T^-$ of $\omega$ is on the row $j$. Then
by the construction
of the virtual tower diagram, we have $j= \eta_l$ and and by Lemma \ref{lem:noZigzag}, the row $j+1$ is empty.
Moreover the column $i$ is the first non-empty column of the tower diagram $\mathcal T$ and the height of this
tower is $j-x$ where $x$ is the number of empty columns on the left of the $i$-th column.

With these notation, the Rothefication $R_{\tilde\omega}$ of $\tilde\omega$ is obtained from $R_\omega$ by
removing the cell $(j,i)$ from $R_\omega$ and then moving the rest of the row $j$ to the next row, which was
empty.

On the other hand, by Proposition \ref{pro:actionRothe}, the Rothe diagram $D_{\tilde\omega}$ is obtained
from $D_\omega$ by removing the cell $(j,\omega(j+1))$ and then interchanging the rows $j$ and $j+1$.
Since $D_{\tilde\omega} = R_{\tilde\omega}$, we immediately conclude that the $j$-th row of $D_{\tilde\omega}$,
and hence the $j+1$-st row of $D_\omega$,  are empty. Therefore to finish the proof, we only need to show that the
removed cell $(j,\omega(j+1))$ coincides with the removed cell $(j,i)$, that is, we need to show that $i= \omega
(j+1)$. Indeed, with this equality, reversing the above steps we obtain the desired equality.

To prove $i = \omega(j+1)$, note that since $i$ is the first non-empty column of $\mathcal T$, we have
$\omega(a) = a$ for any $a<i$. Therefore the first non-empty row of the virtual tower diagram $\mathcal T^-$ is
the $i$-th row. Thence, without loss of generality, we can assume that $i=1$ and hence the height of the
tower at $1$ is $j$ and we are to prove that $\omega(j+1) = 1$. In this case, the natural word of $\omega$ contains
only one copy of the letter $1$ and it ends with the sequence $1\, 2\, \ldots j$. This means that the number $1$ is
moved only by the sequence $1\, 2\, \ldots j$ and hence is at the $j+1$-st place, that is, $\omega(j+1) = 1$, as
required.
\end{proof}
\appendix
\section{Proofs of  the technical lemmas}
\label{Appendix}
\begin{proof} (of Lemma \ref{lem:ax-zigzag})
Let $c_1$ be the lower cell. If the flight path of $c_1$ has only one element, then
the flight path of $c_2$ also has only element. Therefore the flight numbers of $c_1$
and $c_2$ are just the sum of their coordinates and hence the result follows
directly from our assumption.

Otherwise, let $\mathcal T_\star$ be the first tower on the left of $\mathcal T_i$ to which one of $c_1$ or $c_2$
hits. Then there are two possible cases illustrated with the following pictures.

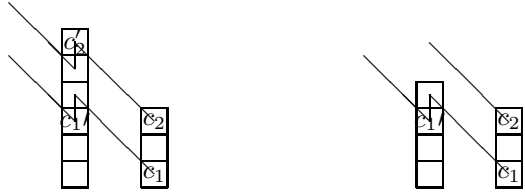
\begin{center}
\begin{picture}(80,60)

\put(20,50){\tableau{{c_2^\prime}\\{}\\{}\\{c_1\prime}\\{}\\{}}}
\put(50,20){\tableau{{c_2}\\{}{}\\{c_1}}}
\multiput(0,50)(0,0){1}{\line(1,-1){25}}
\multiput(25,25)(0,0){1}{\line(0,1){10}}
\multiput(25,35)(0,0){1}{\line(1,-1){30}}

\multiput(0,70)(0,0){1}{\line(1,-1){25}}
\multiput(25,45)(0,0){1}{\line(0,1){10}}
\multiput(25,55)(0,0){1}{\line(1,-1){30}}

\end{picture}\hskip.75in
\begin{picture}(80,60)

\put(20,30){\tableau{{}\\{c_1\prime}\\{}\\{}}}
\put(50,20){\tableau{{c_2}\\{}{}\\{c_1}}}
\multiput(0,50)(0,0){1}{\line(1,-1){25}}
\multiput(25,25)(0,0){1}{\line(0,1){10}}
\multiput(25,35)(0,0){1}{\line(1,-1){30}}

\multiput(25,55)(0,0){1}{\line(1,-1){30}}

\end{picture}\hskip.15in
\end{center}

In the first case, both $c_1$ and $c_2$ make zigzag at $\mathcal T_\star$ and hence
visit the cells $c_1^\prime$ and $c_2^\prime$. Observe that, in this case, the flight
number of $c_i^\prime$ is equal to that of $c_i$ for $i=1,2$. Thus the result follows
from induction on the number of towers on the left of the tower we begin with.

In the second case, the cell $c_1$ makes zigzag and $c_2$ passes through without a
zigzag. In this case, we may replace the tower $\mathcal T_\star$ of $\mathcal T$ with
$T_\star^\prime$ given by the following picture.
\begin{center}
\begin{picture}(80,60)
\put(20,50){\tableau{{c_2^\prime}\\{}\\{}\\{c_1\prime}\\{}}}
\end{picture}\hskip.15in
\end{center}
Now the new tableau $\mathcal T^\prime$ has the cells $c_1^\prime$ and $c_2^\prime$
with flight paths. Here the flight numbers satisfies
\[
\mathrm{flight}\#(c_1,\mathcal T) = \mathrm{flight}\#(c_1^\prime,\mathcal T) =
\mathrm{flight}\#(c_1\prime,\mathcal T^\prime)
\]
and
\[
\mathrm{flight}\#(c_2,\mathcal T) =
\mathrm{flight}\#(c_2\prime,\mathcal T^\prime).
\]
Moreover the distance between the cells $c_1^\prime$ and $c_2^\prime$ is $\mid j_1-j_2 \mid$. Again the result
follows from induction on the number of towers on the left of the tower.

Finally, in the special case where $\mid j_1-j_2\mid = 1$, the second case above will never appear. Therefore
the difference between the flight numbers of $c_1$ and $c_2$ will never change, as required.
\end{proof}

\begin{proof}(of Lemma \ref{Lemma:SlidingFarAway}) Without lost of generality we assume that $\beta\geq \alpha+2$.
First consider the slides $\beta^{\searrow}
(\alpha^{\searrow}\mathcal{T})$.

\vskip.1in\noindent \textit{Case {\bf S1}.} We  first assume that
$\mathcal{T}$ has  no squares lying on the diagonal $x+y=\alpha-1$.
Then by the definition, we have  $\alpha^{\searrow} \mathcal{T}:=
(\mathcal{T}_1,\ldots,\mathcal{T}_{\alpha-1}) \sqcup
\alpha^{\searrow} (\mathcal{T}_{\alpha},\ldots)$.

\vskip.1in\noindent \textit{Case {\bf S1}(a).} If  $\mathcal{T}$ has
no squares lying on the diagonal $x+y=\alpha$ then necessarily
$(\alpha,0)\not \in \mathcal{T}_\alpha$ and  we have
$\alpha^{\searrow}
(\mathcal{T}_{\alpha},\ldots):=(\mathcal{T}'_\alpha,\ldots,
\mathcal{T}_k) $ where  $\mathcal{T}_\alpha'=\{ (\alpha,0)\}$ and it
does not contain any square lying on the diagonal $x+y=\beta-1$.
Hence
\begin{align*}
\beta^{\searrow}(\alpha^{\searrow}
(\mathcal{T}_{1},\ldots,\mathcal{T}_{\alpha},\ldots))&=\beta^{\searrow}
(\mathcal{T}_{1},\ldots,\mathcal{T}_{\alpha}',\ldots)\\
&=(\mathcal{T}_{1},\ldots,\mathcal{T}_{\alpha}') \sqcup
\beta^{\searrow}(\mathcal{T}_{\alpha+1},\ldots).
\end{align*}
On the other hand, we have
\begin{align*}
\alpha^{\searrow}(\beta^{\searrow}
(\mathcal{T}_{1},\ldots,\mathcal{T}_{\alpha},\ldots))&=\alpha^{\searrow}(
(\mathcal{T}_{1},\ldots,\mathcal{T}_{\alpha})\sqcup\beta^{\searrow}(
\mathcal{T}_{\alpha+1},\ldots))\\
&=(\mathcal{T}_{1},\ldots,\mathcal{T}_{\alpha}') \sqcup
\beta^{\searrow}(\mathcal{T}_{\alpha+1},\ldots).
\end{align*}
Comparing the last two equations, we see that the required equality
is satisfied.

\vskip.1in\noindent \textit{Case {\bf S1}(b).} If $(\alpha,0) \in
\mathcal{T}_\alpha$ and $(\alpha,1)\not \in \mathcal{T}_\alpha$ then
the slide $\alpha^{\searrow} \mathcal{T}$ terminates without a
result and so the slide $\beta^{\searrow}(\alpha^{\searrow}
\mathcal{T})$ also terminates. Observe that, in this case,
$\mathcal{T}_\alpha$ has no square lying on the diagonal
$x+y=\beta-1$ and this yields the equality
$$\alpha^{\searrow}(\beta^{\searrow} (\mathcal{T}_{1},\ldots,\mathcal{T}_{\alpha},\ldots))=
\alpha^{\searrow}(
(\mathcal{T}_{1},\ldots,\mathcal{T}_{\alpha})\sqcup
\beta^{\searrow}(\mathcal{T}_{\alpha+1},\ldots)).
$$
On the other hand   the assumption on the integer $\alpha$ and the
tower $\mathcal{T}_{\alpha}$ forces the slide $\alpha^{\searrow}(
\mathcal{T}_{1},\ldots,\mathcal{T}_{\alpha})$ to terminate.
Therefore  $\alpha^{\searrow}(\beta^{\searrow} \mathcal{T})$
terminates as required.

\vskip.1in\noindent \textit{Case {\bf S1}(c).} If $(\alpha,0) \in
\mathcal{T}_\alpha$ and $(\alpha,1) \in \mathcal{T}_\alpha$ then we
have
$$\alpha^{\searrow} (\mathcal{T}_1,\ldots, \mathcal{T}_{\alpha},\ldots)=
(\mathcal{T}_1,\ldots,\mathcal{T}_\alpha) \sqcup
(\alpha+1)^{\searrow} (\mathcal{T}_{\alpha+1},\ldots).
$$ There are two cases to consider.

\vskip.1in \noindent \textit{Case {\bf S1}(c)(i).} If the tower
$\mathcal{T}_\alpha$ has no square on the diagonal $x+y=\beta-1$
then we have the equality
\begin{align*}
\beta^{\searrow}(\alpha^{\searrow}
\mathcal{T})&=\beta^{\searrow}((\mathcal{T}_1,\ldots,
\mathcal{T}_\alpha) \sqcup (\alpha+1)^{\searrow}
(\mathcal{T}_{\alpha+1},\ldots))\\
&= (\mathcal{T}_1,\ldots, \mathcal{T}_\alpha) \sqcup
\beta^{\searrow}((\alpha+1)^{\searrow}
(\mathcal{T}_{\alpha+1},\ldots)).
\end{align*}
On the other hand, we have
\begin{align*}\alpha^{\searrow}(\beta^{\searrow} \mathcal{T})&=\alpha^{\searrow}((\mathcal{T}_1,\ldots, \mathcal{T}_\alpha) \sqcup
\beta^{\searrow}(\mathcal{T}_{\alpha+1},\ldots))\\&=((\mathcal{T}_1,\ldots,
\mathcal{T}_\alpha) \sqcup
(\alpha+1)^{\searrow}(\beta^{\searrow}(\mathcal{T}_{\alpha+1},\ldots))).
\end{align*}

Recall that $(\alpha,1) \in \mathcal{T}_\alpha$ but
$\mathcal{T}_\alpha$ has no square on the diagonal $x+y=\beta-1$,
i.e., $(\alpha,\beta-\alpha-1)\not \in \mathcal{T}_\alpha$. This
show that $\beta-\alpha-1>1$ and hence $\beta-(\alpha+1) \geq 2$.

On the other hand  the number of cells in
$(\mathcal{T}_{\alpha+1},\ldots)$ is strictly less than that in
$\mathcal{T}$. Therefore,  by induction on the number of cells in a
tower diagram, we can assume that either
$\beta^{\searrow}((\alpha+1)^{\searrow}
(\mathcal{T}_{\alpha+1},\ldots))=(\alpha+1)^{\searrow}(\beta^{\searrow}
(\mathcal{T}_{\alpha+1},\ldots)$ or  both slides terminate since
$\beta-(\alpha+1) \geq 2$. Lastly, comparing
$\alpha^{\searrow}(\beta^{\searrow} \mathcal{T})$ and
 $\beta^{\searrow}(\alpha^{\searrow} \mathcal{T})$
we see that  either they are equal or both of them terminate.

\vskip.1in \noindent \textit{Case {\bf S1}(c)(ii).} Now we suppose
that the tower $\mathcal{T}_\alpha$ has a square on the diagonal
$x+y=\beta-1$. Then this square must be $(\alpha,\beta-\alpha-1)$.
If $(\alpha,\beta-\alpha)\not \in \mathcal{T}_\alpha$ then we have
\begin{align*}
\beta^{\searrow}(\alpha^{\searrow} \mathcal{T})&=\beta^{\searrow}(
(\mathcal{T}_1,\ldots, \mathcal{T}_\alpha) \sqcup
(\alpha+1)^{\searrow} (\mathcal{T}_{\alpha+1},\ldots))\\ &=
(\mathcal{T}_1,\ldots, \mathcal{T}_\alpha') \sqcup
(\alpha+1)^{\searrow} (\mathcal{T}_{\alpha+1},\ldots))
\end{align*}
where $\mathcal{T}_\alpha'= \mathcal{T}_\alpha
\cup\{(\alpha,\beta-\alpha)\}$. On the other hand, we have
\begin{align*}
\alpha^{\searrow}(\beta^{\searrow} \mathcal{T})&=\alpha^{\searrow}(
(\mathcal{T}_1,\ldots
,\mathcal{T}_\alpha',\mathcal{T}_{\alpha+1},\ldots)\\
&= (\mathcal{T}_1,\ldots, \mathcal{T}_\alpha') \sqcup
(\alpha+1)^{\searrow} (\mathcal{T}_{\alpha+1},\ldots)).
\end{align*}
Thus if the slide $(\alpha+1)^{\searrow}
(\mathcal{T}_{\alpha+1},\ldots)$ terminates then both of the slides
$\alpha^{\searrow}(\beta^{\searrow} \mathcal{T})$ and
$\beta^{\searrow}(\alpha^{\searrow} \mathcal{T})$ terminate, and
otherwise the resulting diagrams are the same.

Now if $(\alpha,\beta-\alpha) \in \mathcal{T}_\alpha$ but
$(\alpha,\beta-\alpha+1)\not \in \mathcal{T}_\alpha$ then the slide
$\beta^{\searrow}\mathcal{T}$ terminates and therefore the slide
$\alpha^{\searrow}(\beta^{\searrow}\mathcal{T})$ also terminates. On
the other hand, we have
$$\beta^{\searrow}(\alpha^{\searrow} \mathcal{T})=\beta^{\searrow}(
(\mathcal{T}_1,\ldots, \mathcal{T}_\alpha) \sqcup
(\alpha+1)^{\searrow} (\mathcal{T}_{\alpha+1},\ldots))
$$
and  it also  terminates by the assumption on $\beta$ and
$\mathcal{T}_\alpha$.

Lastly we assume that $(\alpha,\beta-\alpha) \in \mathcal{T}_\alpha$
and $(\alpha,\beta-\alpha+1) \in \mathcal{T}_\alpha$. Then we have
$$ \begin{aligned}
\beta^{\searrow}(\alpha^{\searrow} \mathcal{T})=&\beta^{\searrow}(
(\mathcal{T}_1,\ldots, \mathcal{T}_\alpha) \sqcup
(\alpha+1)^{\searrow} (\mathcal{T}_{\alpha+1},\ldots))\\=&
(\mathcal{T}_1,\ldots, \mathcal{T}_\alpha) \sqcup
(\beta+1)^{\searrow}((\alpha+1)^{\searrow}
(\mathcal{T}_{\alpha+1},\ldots)) \end{aligned}$$ and
$$ \begin{aligned} \alpha^{\searrow}(\beta^{\searrow} \mathcal{T})=&\alpha^{\searrow}(
(\mathcal{T}_1,\ldots, \mathcal{T}_\alpha) \sqcup
(\beta+1)^{\searrow} (\mathcal{T}_{\alpha+1},\ldots))\\ =&
(\mathcal{T}_1,\ldots, \mathcal{T}_\alpha) \sqcup
(\alpha+1)^{\searrow}((\beta+1)^{\searrow}
(\mathcal{T}_{\alpha+1},\ldots)).
\end{aligned}$$
It is clear that $(\beta+1)-(\alpha+1)\geq 2$ and thus, by induction
on the number of cells in a tower diagram, we have either the
equality
$$(\beta+1)^{\searrow}((\alpha+1)^{\searrow}
(\mathcal{T}_{\alpha+1},\ldots))=(\alpha+1)^{\searrow}((\beta+1)^{\searrow}
(\mathcal{T}_{\alpha+1},\ldots))$$ or that both slides terminate.
Therefore either $\alpha^{\searrow}(\beta^{\searrow} \mathcal{T})$
and  $\beta^{\searrow}(\alpha^{\searrow} \mathcal{T})$ are equal or
both terminate as required.

\vskip.1in \noindent \textit{Case {\bf S2}.} Let $(i,\alpha-1-i)\in
\mathcal{T}_i$ be the first square of $\mathcal{T}$, from the left,
lying on the diagonal $x+y=\alpha-1$. We have the following
sub-cases:

\vskip.1in \noindent \textit{Case {\bf S2}(a).} If $(i,\alpha-i)\not
\in \mathcal{T}_i$ then we have
$$\alpha^{\searrow} \mathcal{T}:=
(\mathcal{T}_1,\ldots,\mathcal{T}_{i-1}) \sqcup \alpha^{\searrow}
(\mathcal{T}_{i},\ldots)=(\mathcal{T}_1,\ldots,\mathcal{T}_{i-1},
 \mathcal{T}'_{i},\mathcal{T}_{i+1},\ldots)$$
where $\mathcal{T}_i'=\mathcal{T}_i \cup \{ (i,\alpha-i)\}$. Now
since $\beta\geq \alpha+2$, none of the towers
$\mathcal{T}_1,\ldots,\mathcal{T}_{i-1}, \mathcal{T}'_{i}$ contains
a square on the diagonal $x+y=\beta-1$ and hence we obtain the
equality
\begin{align*}
\beta^{\searrow}(\alpha^{\searrow} \mathcal{T})&=
\beta^{\searrow}(\mathcal{T}_1,\ldots,\mathcal{T}_{i-1}
 \mathcal{T}'_{i},\mathcal{T}_{i+1},\ldots)\\ &=
(\mathcal{T}_1,\ldots,\mathcal{T}'_{i}) \sqcup \beta^{\searrow}
(\mathcal{T}_{i+1},\ldots).
\end{align*}

\noindent On the other hand, we have
\begin{align*}
\alpha^{\searrow}(\beta^{\searrow} \mathcal{T})&:=\alpha^{\searrow}
((\mathcal{T}_1,\ldots,\mathcal{T}_{i}) \sqcup \beta^{\searrow}
(\mathcal{T}_{i+1},\ldots))\\
&=(\mathcal{T}_1,\ldots,\mathcal{T}'_{i}) \sqcup \beta^{\searrow}
(\mathcal{T}_{i+1},\ldots).
\end{align*}

Hence if the slide $\beta^{\searrow} (\mathcal{T}_{i+1},\ldots)$
 terminates then both  $\alpha^{\searrow}(\beta^{\searrow}
\mathcal{T})$ and $\beta^{\searrow}(\alpha^{\searrow} \mathcal{T})$
terminate, and otherwise the resulting diagrams are the same.

\vskip.1in \noindent \textit{Case {\bf S2}(b).} If $(i,\alpha-i) \in
\mathcal{T}_i$ and $(i,\alpha-i+1)\not \in \mathcal{T}_i$ then the
slides $\alpha^{\searrow} (\mathcal{T}_{i},\ldots)$ and
$\alpha^{\searrow} \mathcal{T}$ terminate without a result.
Therefore the slide $\beta^{\searrow}(\alpha^{\searrow}
\mathcal{T})$ also terminates. Now if $\beta^{\searrow} \mathcal{T}$
terminates then the slide $\alpha^{\searrow}(\beta^{\searrow}
\mathcal{T})$ also terminates. Otherwise, we have
\begin{align*}
\beta^{\searrow} \mathcal{T}&=
(\mathcal{T}_1,\ldots,\mathcal{T}_{i}) \sqcup \beta^{\searrow}
(\mathcal{T}_{i+1},\ldots)\\
\alpha^{\searrow}(\beta^{\searrow}
\mathcal{T})&=\alpha^{\searrow}((\mathcal{T}_1,\ldots,\mathcal{T}_{i})
\sqcup \beta^{\searrow} (\mathcal{T}_{i+1},\ldots))
\end{align*}
but still   $\alpha^{\searrow}(\beta^{\searrow} \mathcal{T})$
 terminates since $(i,\alpha-i) \in
\mathcal{T}_i$ but  $(i,\alpha-i+1)\not \in \mathcal{T}_i$.

\vskip.1in \noindent \textit{Case {\bf S2}(c). } If $(i,\alpha-i) \in
\mathcal{T}_i$ and $(i,\alpha-i+1) \in \mathcal{T}_i$  then we have
$$\alpha^{\searrow} \mathcal{T}:= (\mathcal{T}_1,\ldots, \mathcal{T}_i) \sqcup
(\alpha+1)^{\searrow} (\mathcal{T}_{i+1},\ldots).
$$
There are two more cases to consider.

\vskip.1in \noindent \textit{Case {\bf S2}©(i).} If $\mathcal{T}_i$
has no square on the diagonal $x+y=\beta-1$ then, we have
$$\beta^{\searrow}(\alpha^{\searrow} \mathcal{T})= (\mathcal{T}_1,\ldots, \mathcal{T}_i) \sqcup
\beta^{\searrow}((\alpha+1)^{\searrow} (\mathcal{T}_{i+1},\ldots)).
$$
On the other hand, we have
\begin{align*}
\alpha^{\searrow}(\beta^{\searrow}
\mathcal{T})&:=\alpha^{\searrow}((\mathcal{T}_1,\ldots,
\mathcal{T}_i) \sqcup
\beta^{\searrow}(\mathcal{T}_{i+1},\ldots))\\
&=((\mathcal{T}_1,\ldots, \mathcal{T}_i) \sqcup
(\alpha+1)^{\searrow}(\beta^{\searrow}(\mathcal{T}_{i+1},\ldots))).
\end{align*}

Recall that $(i,\alpha-i) \in \mathcal{T}_i$ but $\mathcal{T}_i$ has
no square on the diagonal $x+y=\beta-1$, i.e., $(i,\beta-i-1)\not
\in \mathcal{T}_\alpha$. This show that $\beta-i-1>\alpha-i$ and
hence $\beta-(\alpha+1) \geq 2$.

On the other hand  the number of cells in
$(\mathcal{T}_{i+1},\ldots)$ is strictly less than that of
$\mathcal{T}$. Therefore,  by induction on the number of cells in a
tower diagram, we can assume that either
$\beta^{\searrow}((\alpha+1)^{\searrow}
(\mathcal{T}_{i+1},\ldots))=(\alpha+1)^{\searrow}(\beta^{\searrow}
(\mathcal{T}_{i+1},\ldots)$ or  both slides terminate since
$\beta-(\alpha+1) \geq 2$. Lastly, comparing
$\alpha^{\searrow}(\beta^{\searrow} \mathcal{T})$ and
 $\beta^{\searrow}(\alpha^{\searrow} \mathcal{T})$
we see that  either they are equal or both of them terminate.

\vskip.1in \noindent \textit{Case {\bf S2}(c)(ii).} Now we suppose
that  $\mathcal{T}_i$ has a square on the diagonal $x+y=\beta-1$.
Then this square must be $(i,\beta-i-1)$.

If $(i,\beta-i)\not \in \mathcal{T}_i$ then
\begin{align*}
\beta^{\searrow}(\alpha^{\searrow} \mathcal{T})&=\beta^{\searrow}(
(\mathcal{T}_1,\ldots, \mathcal{T}_i) \sqcup (\alpha+1)^{\searrow}
(\mathcal{T}_{i+1},\ldots))\\ &= (\mathcal{T}_1,\ldots,
\mathcal{T}_i') \sqcup (\alpha+1)^{\searrow}
(\mathcal{T}_{i+1},\ldots))
\end{align*}
where $\mathcal{T}_i'= \mathcal{T}_i \cup\{(i,\beta-i)\}$. On the
other hand, we have
\begin{align*}
\alpha^{\searrow}(\beta^{\searrow} \mathcal{T})&=\alpha^{\searrow}(
(\mathcal{T}_1,\ldots, \mathcal{T}_i',\mathcal{T}_{i+1},\ldots)\\
&= (\mathcal{T}_1,\ldots, \mathcal{T}_i') \sqcup
(\alpha+1)^{\searrow} (\mathcal{T}_{i+1},\ldots)).
\end{align*}

Now if the slide $(\alpha+1)^{\searrow} (\mathcal{T}_{i+1},\ldots)$
terminates then both of the slides
$\alpha^{\searrow}(\beta^{\searrow} \mathcal{T})$ and
$\beta^{\searrow}(\alpha^{\searrow} \mathcal{T})$ terminate, and
otherwise the resulting diagrams are the same.

If $(i,\beta-i) \in \mathcal{T}_i$ but $(i,\beta-i+1)\not \in
\mathcal{T}_i$ then the slide $\beta^{\searrow}\mathcal{T}$
terminates and therefore the slide
$\alpha^{\searrow}(\beta^{\searrow}\mathcal{T})$ also terminates. On
the other hand, we have
$$\beta^{\searrow}(\alpha^{\searrow} \mathcal{T})=\beta^{\searrow}(
(\mathcal{T}_1,\ldots, \mathcal{T}_i) \sqcup (\alpha+1)^{\searrow}
(\mathcal{T}_{i+1},\ldots))
$$
and it also terminates by the assumption on $\beta$ and
$\mathcal{T}_i$.

For the last case, we assume that $(i,\beta-i)$ and $(i,\beta-i+1)$
are in $\mathcal{T}_i$. Then we have
$$ \begin{aligned}
\beta^{\searrow}(\alpha^{\searrow} \mathcal{T})=&\beta^{\searrow}(
(\mathcal{T}_1,\ldots, \mathcal{T}_i) \sqcup (\alpha+1)^{\searrow}
(\mathcal{T}_{i+1},\ldots))\\=& (\mathcal{T}_1,\ldots,
\mathcal{T}_i) \sqcup (\beta+1)^{\searrow}((\alpha+1)^{\searrow}
(\mathcal{T}_{i+1},\ldots)) \end{aligned}$$
 and also
$$ \begin{aligned} \alpha^{\searrow}(\beta^{\searrow} \mathcal{T})=&\alpha^{\searrow}(
(\mathcal{T}_1,\ldots, \mathcal{T}_i) \sqcup (\beta+1)^{\searrow}
(\mathcal{T}_{i+1},\ldots))\\ =&
(\mathcal{T}_1,\ldots, \mathcal{T}_i) \sqcup
(\alpha+1)^{\searrow}((\beta+1)^{\searrow}
(\mathcal{T}_{i+1},\ldots)).
\end{aligned}$$
Thus, by induction, we have either the equality
$$(\beta+1)^{\searrow}((\alpha+1)^{\searrow}
(\mathcal{T}_{i+1},\ldots))=(\alpha+1)^{\searrow}((\beta+1)^{\searrow}
(\mathcal{T}_{i+1},\ldots))$$ or that both slides terminate and this
gives the required result.
\end{proof}

\begin{proof}(of Lemma \ref{Lemma:SlidingClose})  The case that
 $\mathcal{T}$ has  no
squares lying on the diagonal $x+y=\alpha-1$ ({\bf S1})  can be
dealt with in the same manner as the case  that $\mathcal{T}$ has
some squares lying on the diagonal $x+y=\alpha-1$ ({\bf S2}), as illustrated in the proof of the previous Lemma.
Because of this reason in the following we will just work on the
case of {\bf S2}. Therefore, let $\mathcal{T}_i$ be the
 be the first tower from the left which contains a square, necessarily $(i,\alpha-1-i)$,  on the diagonal $x+y=\alpha-1$.

\vskip.1in \noindent \textit{Case 1.} If $(i,\alpha-i)\not \in
\mathcal{T}_i$ then
$\alpha^{\searrow} \mathcal{T}=(\mathcal{T}_1,\ldots,\mathcal{T}_{i-1},
 \mathcal{T}'_{i},\mathcal{T}_{i+1},\ldots)$
where $\mathcal{T}_i'=\mathcal{T}_i \cup \{ (i,\alpha-i)\}$. Now
since the cell $\{ (i,\alpha-i)\}$ is the first cell of $\mathcal{T}$, from the left,
lying on the diagonal $x+y=(\alpha+1)-1$, we have that
$$(\alpha+1)^{\searrow}(\alpha^{\searrow} \mathcal{T})=(\mathcal{T}_1,\ldots,\mathcal{T}_{i-1},
 \mathcal{T}''_{i},\mathcal{T}_{i+1},\ldots)$$
where $\mathcal{T}_i''=\mathcal{T}_i \cup  \{ (i,\alpha-i),(i,\alpha-i+1)\}$. On the other hand, since $\mathcal{T}_i''$ is still the first tower
containing a cell on $x+y=\alpha-1$ and since $\{ (i,\alpha-i),(i,\alpha-i+1)\}\subset \mathcal{T}_i''$ we have
\begin{align*}
\alpha^{\searrow}((\alpha+1)^{\searrow}(\alpha
^{\searrow}\mathcal{T}))&=\alpha^{\searrow}(\mathcal{T}_1,\ldots,\mathcal{T}_{i-1},
 \mathcal{T}''_{i},\mathcal{T}_{i+1},\ldots)\\ &=(\mathcal{T}_1,\ldots,\mathcal{T}_{i-1},
 \mathcal{T}''_{i}) \sqcup (\alpha+1)^{\searrow}(\mathcal{T}_{i+1},\ldots).
\end{align*}
Observe that since none of the towers $ \mathcal{T}_1, \mathcal{T}_2, \ldots \mathcal{T}_i$ of $\mathcal{T}$ contains a cell on the diagonal $x+y=(\alpha+1)-1$, we get
$$(\alpha+1)^{\searrow} \mathcal{T}=(\mathcal{T}_1,\ldots,\mathcal{T}_{i})\sqcup (\alpha+1)^{\searrow} (\mathcal{T}_{i+1}, \ldots).
$$
Moreover the assumption that $(i,\alpha-i)\not \in
\mathcal{T}_i$  yields the equality
$$(\alpha+1)^{\searrow}(\alpha^{\searrow}((\alpha+1)^{\searrow} \mathcal{T})=(\mathcal{T}_1,\ldots,\mathcal{T}''_{i})\sqcup (\alpha+1)^{\searrow} (\mathcal{T}_{i+1}, \ldots)
$$
where $\mathcal{T}_i''=\mathcal{T}_i \cup  \{ (i,\alpha-i),(i,\alpha-i+1)\}$. Hence we have the desired result.

\vskip.1in \noindent \textit{Case 2.} If $(i,\alpha-i) \in
\mathcal{T}_i$ and $(i,\alpha+1-i)\not \in \mathcal{T}_i$ then the slide $\alpha^{\searrow} \mathcal{T}$ (and therefore the slide $\alpha^{\searrow}(\alpha+1^{\searrow}(\alpha^{\searrow} \mathcal{T}))$) terminates without a result.
 On the other hand since $\mathcal{T}_i$ is the first tower containing a cell on $x+y=(\alpha+1)-1$ and since  $(i,\alpha+1-i)\not \in \mathcal{T}_i$ we have
$$(\alpha+1)^{\searrow} \mathcal{T})=(\mathcal{T}_1,\ldots,\mathcal{T}'_{i}, \ldots)$$
 where $\mathcal{T}'_{i}=\mathcal{T}_{i}\cup \{(i,\alpha+1-i)\}$. Now
$$\alpha^{\searrow}((\alpha+1)^{\searrow} \mathcal{T}))=\alpha^{\searrow}(\mathcal{T}_1,\ldots,\mathcal{T}'_{i}, \ldots)=(\mathcal{T}_1,\ldots,\mathcal{T}'_{i})\sqcup (\alpha+1)^{\searrow}
(\mathcal{T}_{i+1},\ldots)$$ since $\mathcal{T}'_{i}$ contains both $(i,\alpha-i)$ and $(i,\alpha+1-i)$. On the other hand  $(i,\alpha+2-i) \not \in \mathcal{T}'_{i}$. Thus the slide $(\alpha+1)^{\searrow} (\mathcal{T}_1,\ldots,\mathcal{T}'_{i})\sqcup (\alpha+1)^{\searrow}
(\mathcal{T}_{i+1},\ldots)$ terminates and hence the slide $(\alpha+1)^{\searrow}(\alpha^{\searrow}((\alpha+1)^{\searrow} \mathcal{T}))$ also terminates as required.

\vskip.1in \noindent \textit{Case 3.} If $(i,\alpha-i) \in
\mathcal{T}_i$ and $(i,\alpha+1-i) \in \mathcal{T}_i$  then
$\alpha^{\searrow} \mathcal{T}= (\mathcal{T}_1\ldots \mathcal{T}_i) \sqcup
(\alpha+1)^{\searrow} (\mathcal{T}_{i+1},\ldots)$.

We first suppose that  $(i,\alpha+2-i) \in \mathcal{T}_i$. Then, we have
$$\begin{aligned}
\alpha^{\searrow}(\alpha+1^{\searrow})(\alpha^{\searrow} \mathcal{T})=&\alpha^{\searrow}(\alpha+1^{\searrow})((\mathcal{T}_1,\ldots, \mathcal{T}_i) \sqcup
(\alpha+1)^{\searrow} (\mathcal{T}_{i+1},\ldots))\\
=&\alpha^{\searrow} ((\mathcal{T}_1,\ldots, \mathcal{T}_i) \sqcup
(\alpha+2)^{\searrow}(\alpha+1)^{\searrow} (\mathcal{T}_{i+1},\ldots))\\
=&(\mathcal{T}_1,\ldots, \mathcal{T}_i) \sqcup
(\alpha+1)^{\searrow}(\alpha+2)^{\searrow}(\alpha+1)^{\searrow} (\mathcal{T}_{i+1},\ldots)\\
\end{aligned}
$$
On the other hand
$$
(\alpha+1^{\searrow})(\alpha^{\searrow}(\alpha+1^{\searrow}) \mathcal{T})
=(\mathcal{T}_1,\ldots, \mathcal{T}_i) \sqcup
(\alpha+2)^{\searrow}(\alpha+1)^{\searrow}(\alpha+2)^{\searrow} (\mathcal{T}_{i+1},\ldots).\\
$$
Therefore, an induction argument  on the number of towers gives the required result.

Next suppose that   $(i,\alpha+2-i) \not \in \mathcal{T}_i$. Then   we have $\alpha^{\searrow} \mathcal{T}=
(\mathcal{T}_1\ldots \mathcal{T}_i) \sqcup (\alpha+1)^{\searrow} (\mathcal{T}_{i+1},\ldots)$. Now the fact that $
\mathcal{T}_i$ is the first tower of  $\mathcal{T}$ containing  $(i,\alpha-i)$ on the diagonal $x+y=(\alpha+1)-1$ and
the fact that  $(i,\alpha+1-i)\in\mathcal{T}_i $ but $(i,\alpha+2-i)\not \in\mathcal{T}_i $ gives that the slide $(\alpha
+1^{\searrow})(\alpha^{\searrow}\mathcal{T})$ and that the slide $\alpha+1^{\searrow}\mathcal{T}$ terminate.
Therefore both the slide $\alpha^{\searrow}(\alpha+1^{\searrow}(\alpha^{\searrow} \mathcal{T})$ and the slide $
\alpha+1^{\searrow}(\alpha^{\searrow}
(\alpha+1^{\searrow} \mathcal{T}))$ terminate, as required.
\end{proof}

\begin{proof}(of Lemma \ref{lem:noZigzag}) Let $D$ be the cell just below the
cell $C$. Since the cell $C$ is already filled, the cell $D$ can only be filled by a
zigzag. We show that such a zigzag cannot exist in the virtual sliding of the natural
word.

By its definition, the natural word $\eta$ of $\omega$ is a
sequence of strictly increasing sequences $\lambda_1,
\lambda_2,\ldots,\lambda_k$ such that the subsequence of initial
terms of $\lambda_i$ is strictly decreasing. Therefore the reverse
word $\tilde\eta$ is a sequence of strictly decreasing sequences
$\tilde\lambda_k,\tilde\lambda_{k-1},\ldots,\tilde\lambda_1$ such
that the subsequence of the terminal terms is strictly increasing.
In particular, we observe that if $j$ is a terminal term for a
subsequence $\tilde\lambda_l$, then this is the last occurrence of
$j$ in $\eta$.

Observe also that the first block $B$ of towers in the virtual tower, from
top to bottom, is the block containing the cell $C$. (Here by a block, we mean a
connected component of the tableau $T^-$.)
Indeed the cell $C$ corresponds to the sliding of the first letter of
$\tilde\lambda_k$ and the block $B$ has top cell corresponding to the last letter
of $\tilde\lambda_k$. Thus as well as the last letter, there can appear no smaller letter. But to have a new block on the top of the block $B$, there is need for a
slide of a smaller letter. Thus $B$ is the first block.

Notice that the
same argument also proves that the cell $D$ should be empty. Indeed since there
is no tower on the top of $B$, the only way to fill $D$ is a slide of the last
letter of $\tilde\lambda_k$. But this letter cannot appear again.
\end{proof}


\begin{thebibliography}{EMG}
\bibitem{BB1} N.~Bergeron, S.~Billey, \emph{RC-graphs and Schubert polynomials},
Experiment Math. \textbf{2}, (1993), 257-269.
\bibitem{BB} A.~Björner, F.~Brenti, \emph{Combinatorics of Coxeter groups}, Springer (2005).
\bibitem{EG} P.~Edelman, C.~Greene, \emph{Balanced Tableaux}, Adv. in Math. \textbf{63} (1987), 42-99.
\bibitem{FGRS} S.~Fomin, C.~Greene, V.~Reiner, M.~Shimozono, \emph{Balanced
labellings and Schubert polynomials}, European J. Combinatorics, \textbf{18} (1997),
379-389.
\bibitem{K} D.~Knuth, \emph{The art of computer programming,} Vol. III, Addison-Wesley Pub. Co., (1973).
\bibitem{LS} A.~Lascoux, M.P.~Schützenberger, \emph{Structure de Hopf de l'anneau de cohomologie et de l'anneau de Grothendieck d'une variete de drapeaux}, C. R. Acad. Sci., Paris, Ser. I, \text{295} (1982), 629-633.
\bibitem{M} I.~G.~Macdonald, \emph{Notes on Schubert polynomials}, Universite de
Quebec a Montreal, 1991.
\bibitem{RS} V.~Reiner, M.~Shimozono,  \emph{Plactification}, J. Alg. Comb.  \textbf{4} (1995), 331-351.
\bibitem{S} R.~Stanley, \emph{On the number of reduced decompositions of elements of Coxeter groups}, European J. Combinatorics, \text{5} (1984), 359-372.

\end{thebibliography}
\end{document}